\documentclass[12pt]{article}
\usepackage[utf8]{inputenc}
\usepackage{amsmath, amssymb, amsthm, mathtools}
\usepackage{bbold,tabularx}
\usepackage[shortlabels]{enumitem}
\usepackage{float,subcaption,xcolor,multirow}
\usepackage{hyperref,tikz, tikzpagenodes}
\usetikzlibrary{through,intersections}
\usepackage[maxbibnames=9]{biblatex}
\usepackage{tikz-cd}
\mathtoolsset{centercolon}
\usepackage{xurl}
\usepackage{longtable}
\usepackage{comment}

\newtheorem{theorem}{Theorem}
\newtheorem{corollary}[theorem]{Corollary}
\newtheorem{lemma}[theorem]{Lemma}

\theoremstyle{definition}
\newtheorem{definition}[theorem]{Definition}
\newtheorem{example}[theorem]{Example}
\newtheorem{remark}[theorem]{Remark}

\usepackage[left=27mm,right=27mm,top=20mm,bottom=20mm]{geometry}

\def\0{{\bf 0}}
\def\1{{\bf 1}}

\newcommand{\R}{{\mathbb {R}}}

\colorlet{ss}{black!20}
\colorlet{sslabel}{black!50}
\colorlet{backline}{black!10}

\newcommand{\labelj}{j}
\newcommand{\labeljdist}{5pt}

\makeatletter

\newcommand{\rref}[2]{\hyperref[#2]{{#1}~\ref*{#2}}}

\usepackage{xcolor}
\usepackage{eurosym}
\title{Switching methods of level 2 \\for the construction of cospectral graphs}
\author{Aida Abiad\thanks{\texttt{a.abiad.monge@tue.nl}, Department of Mathematics and Computer Science, Eindhoven University of Technology, The Netherlands; Department of Mathematics and Data Science of Vrije Universiteit Brussels, Belgium}
\and Nils van de Berg\thanks{\texttt{n.p.v.d.berg@tue.nl}, Department of Mathematics and Computer Science, Eindhoven University of Technology, The Netherlands} \and Robin Simoens\thanks{\texttt{Robin.Simoens@UGent.be},  Department of Mathematics: Analysis, Logic and Discrete Mathematics, Ghent University, Belgium; Department of Mathematics, Universitat Politècnica de Catalunya, Spain}}

\date{}

\begin{document}

\maketitle

\begin{abstract}
A switching method is a graph operation that results in cospectral graphs (graphs with the same spectrum). Work by Wang and Xu [\emph{Discrete Math.} 310 (2010)] suggests that most cospectral graphs with cospectral complements can be constructed using regular orthogonal matrices of level 2, which has relevance for Haemers' conjecture. We present two new switching methods and several combinatorial and geometrical reformulations of existing switching operations of level 2. We also introduce the concept of reducibility and use it to classify all irreducible switching methods that correspond to a conjugation with a regular orthogonal matrix of level 2 with one non-trivial indecomposable block, up to switching sets of size 12, extending previous results.\\

\noindent \textbf{Keywords:} switching, cospectral graphs, regular orthogonal matrix, level 2\\
\noindent \textbf{MSC:} 05C50
\end{abstract}

\section{Introduction}

An important problem in spectral graph theory is to understand which graphs are determined up to isomorphism by their spectrum (of their adjacency matrix), see the surveys \cite{developments,which}. This problem is motivated by the desire for a practical graph isomorphism test. A major unsolved conjecture in this area due to Haemers, is that ``almost all graphs are uniquely determined by their spectrum''. Brouwer and Spence \cite{cospectral12} provided computational evidence for the conjecture by enumerating all graphs with up to 12 vertices (extending work of Godsil and McKay \cite{enumerationGM} and Haemers and Spence \cite{enumeration}), and observed a decline in the fraction of \emph{cospectral mates} (non-isomorphic graphs with the same spectrum) between 10 and 12 vertices. On the other side, Haemers and Spence \cite{enumeration} established an asymptotic lower bound for the number of cospectral mates. Their key ingredient is the notion of switching.

A \emph{switching method} is an operation on a graph that results in a graph with the same spectrum. For such a method to work, the graph needs a special structure, called a \emph{switching set}. This set of vertices makes it possible to swap some of the edges while preserving the spectrum of the adjacency matrix. While Godsil-McKay (GM) switching \cite{GMswitching} is the oldest and most productive switching method in the literature (see for example \cite{abiad2019graph,abiad2016switched,johnson}), new switching methods have recently appeared, most notably the switching methods introduced by Abiad and Haemers \cite{AHswitching} and Wang-Qiu-Hu (WQH) switching \cite{WQHswitching}. The construction of cospectral mates has multiple purposes: to show that a certain graph class is not determined by its spectrum (see for example \cite{johnson,neumaier}) and thus providing new insights for Haemers' conjecture; to show which properties of a graph cannot be deduced from the spectrum, see for example \cite{abiadzf,LWL2020,BCH2015}; or to construct new strongly regular and distance-regular graphs, see for example \cite{twisted,barwick}, among others.

Level $2$ switching methods, like the ones introduced by Abiad and Haemers \cite{AHswitching}, are motivated by results of Wang and Xu \cite{level2mats} which suggest that most \(\mathbb{R}\)-\emph{cospectral graphs} (cospectral graphs with cospectral complements) can be constructed using regular orthogonal matrices of level \(2\). With this motivation in mind, we push this line of research further by extending results from \cite{AHswitching} and \cite{MAO2023}. In particular, we propose two new switching methods of level 2 to construct \(\mathbb{R}\)-cospectral graphs. We give particular attention to switching methods with small switching sets, since they seem to be the most fruitful for constructions of cospectral graphs, see for example \cite{enumeration} and \cite{countingswitching}. We introduce the notion of irreducible switching methods and give a complete description of all irreducible level 2 switching methods with one non-trivial indecomposable block, for switching sets of sizes 6, 8, 10 and 12. The switching methods of level 2 that were introduced by Abiad and Haemers \cite{AHswitching} are described algebraically, making them difficult to use. Mao, Wang, Liu and Qiu \cite[Section~6]{MAO2023} looked at a first combinatorial description of such methods. In this direction, we present a combinatorial description of Six vertex AH-switching \cite[Section~4]{AHswitching}, and we extend this result to switching sets of sizes 8, 10 and 12. Similarly, we propose geometrical and combinatorial interpretations of Seven vertex AH-switching and Eight vertex AH-switching, respectively. Using this new geometrical interpretation of Seven vertex AH-switching, we obtain a shorter and more intuitive proof of a known cospectrality result on 2-Kneser graphs \cite{johnson}. This result provides the first application to graph theory of the level 2 switching methods proposed in \cite{AHswitching}. We note that, at a similar time and independently, related work on level 2 techniques to construct cospectral graphs has been proposed by Mao and Yan \cite{MaoSimilar}. The switching methods that they describe, bear similarities with ours, and these similarities are mentioned throughout the paper.

This paper is structured as follows. We start with some definitions and a description of the two existing switching methods of general level in Section~\ref{sec:preliminaries}. In Section~\ref{sec:level2}, we state a general approach to describe switching methods of level 2 that was started by Abiad and Haemers \cite{AHswitching} for switching sets of size 6, 7 and 8 and also used by Mao, Wang, Liu and Qiu \cite{MAO2023} to find an infinite family of switching methods, which we call MWLQ-switching or Sun graph switching. Section~\ref{sec:new} presents two new switching methods based on regular orthogonal matrices of level 2. Next, we introduce the notion of reducibility in Section~\ref{sec:reducibility} and use it in Section~\ref{sec:classification} to show that the two new methods, together with GM-switching, Six vertex AH-switching and Seven vertex AH-switching, provide a complete classification of the irreducible switching methods of level 2 with one indecomposable block of size at most 10. Moreover, for those methods with a switching set of size 12, an algebraic classification is given. Finally, the results are summarized in Section~\ref{sec:concludingremarks}, where we also pose some open problems.

\section{Preliminaries}\label{sec:preliminaries}

Throughout the paper, we closely follow the notation from \cite{AHswitching}.
  
We consider simple and loopless graphs. Given a vertex \(v\), we use \(N(v)\) to denote the set of its neighbours. The \emph{(adjacency) spectrum} of a graph is the multiset of eigenvalues of its adjacency matrix. Graphs are considered \emph{cospectral} if they have the same spectrum. Two graphs are said to be \emph{cospectral mates} if they are cospectral and non-isomorphic. Let \(I\) denote the identity matrix and \(J\) the all-ones matrix. Two graphs with adjacency matrices \(A\) and \(A'\) are called \emph{\(\mathbb{R}\)-cospectral} if \(A+rJ\) and \(A'+rJ\) are cospectral for every \(r\in\mathbb{R}\).
An orthogonal matrix is \emph{regular} if each row sums to one.
Johnson and Newman \cite{rcospectral} showed that two graphs are \(\mathbb{R}\)-cospectral if and only if their adjacency matrices are conjugated with a regular orthogonal matrix.

The \emph{level} of a matrix is the smallest positive integer \(\ell\) such that \(\ell\) times the matrix is an integral matrix, or \(\infty\) if it has irrational entries. A matrix is \emph{decomposable} if it can be written as a non-trivial block-diagonal matrix after a certain permutation of the rows and columns; otherwise, it is \emph{indecomposable}.

Next, we provide an overview of the existing switching methods to construct \(\mathbb{R}\)-cospectral graphs. We first describe the ones using regular orthogonal matrices of any level. In Section~\ref{sec:level2}, we introduce the known methods of level 2.

\subsection{Godsil-McKay switching}

The following method for finding cospectral graphs was introduced by Godsil and McKay \cite{GMswitching} in 1982.

\begin{theorem}[GM-switching \cite{GMswitching}]\label{thm:GM}
    Let \(\Gamma\) be a graph and let \(\{C_1,\dots,C_t,D\}\) be a partition of its vertices such that, for all \(i,j\in\{1,\dots,t\}\):
    \begin{enumerate}[(i)]
        \item Every vertex in \(C_i\) has the same number of neighbours in \(C_j\).
        \item Every vertex in \(D\) has \(0\), \(\frac{1}{2}|C_i|\) or \(|C_i|\) neighbours in \(C_i\). 
    \end{enumerate}
    For all \(i\in\{1,\dots,t\}\) and every \(v\in D\) that has exactly \(\frac{1}{2}|C_i|\) neighbours in \(C_i\), swap the adjacencies between \(v\) and \(C_i\). The resulting graph is \(\mathbb{R}\)-cospectral with \(\Gamma\).
\end{theorem}

The GM-switching operation corresponds to a conjugation of the adjacency matrix with the orthogonal matrix $\mathrm{diag}(R_1,\dots,R_t,I)$, where $R_i$ equals the $|C_i|\times|C_i|$ matrix $\frac2{|C_i|}J-I$. Note that any $C_i$ of order $2$ only gives a permutation matrix and is therefore trivial. The simplest non-trivial case has $t=1$ and $|C_1|=4$; this case has actually been the most fruitful in the literature, see for example \cite{abiad2019graph,johnson,abiad2016switched,haemers2010graphs}. Larger switching sets require more conditions on the graph, which intuitively explains the relative effectiveness of small switching sets.
Note that, if \(t\geq2\), condition (i) in Theorem~\ref{thm:GM} is stricter than necessary. There are examples of graphs that do not satisfy it, but where a conjugation with the orthogonal matrix \(\mathrm{diag}(\frac12J-I,\frac12J-I,I)\) leads to an \(\mathbb{R}\)-cospectral graph.

\subsection{Wang-Qiu-Hu switching}

In 2019, Wang, Qiu and Hu \cite{WQHswitching} presented another switching method, which corresponds to a conjugation of the adjacency matrix with the orthogonal matrix $\mathrm{diag}(R_1,\dots,R_t,I)$, where each \(R_i\) is of the form \[
R_i = \begin{bmatrix} I - \frac2{|C_i|}J & \frac2{|C_i|}J \\ \frac2{|C_i|}J & I - \frac2{|C_i|}J \end{bmatrix}.
\]

As illustrated in \cite{johnson,ihringerpavese,ihringer2019new,neumaier}, WQH-switching is also a powerful tool for constructing cospectral graphs in cases where GM-switching fails. 
In combinatorial terms, the method can be described as follows.

\begin{theorem}[WQH-switching \cite{WQHswitching}]
    Let \(\Gamma\) be a graph and let \(\{C_1^{(1)},C_1^{(2)},\dots,C_t^{(1)},C_t^{(2)},D\}\) be a partition of its vertices such that, for all \(i,j\in\{1,\dots,t\}\): 
    \begin{enumerate}[(i)]
        \item \(|C_i^{(1)}| = |C_i^{(2)}|\).
        \item The number \(\begin{cases}|N(v)\cap C_j^{(1)}|-|N(v)\cap C_j^{(2)}|&\text{if }v\in C_i^{(1)}\\|N(v)\cap C_j^{(2)}|-|N(v)\cap C_j^{(1)}|&\text{if }v\in C_i^{(2)}\end{cases}\) is the same for every \(v\in C_i^{(1)}\cup C_i^{(2)}\).
        \item Every vertex in \(D\) has either:
        \begin{enumerate}
            \item \(|C_i^{(1)}|\) neighbours in \(C_i^{(1)}\) and \(0\) neighbours in \(C_i^{(2)}\),
            \item \(0\) neighbours in \(C_i^{(1)}\) and \(|C_i^{(2)}|\) neighbours in \(C_i^{(2)}\),
            \item the same number of neighbours in \(C_i^{(1)}\) as in \(C_i^{(2)}\).
        \end{enumerate}
    \end{enumerate}
    For all \(i\in\{1,\dots,t\}\) and every \(v\in D\) for which (a) or (b) holds, swap the adjacencies between \(v\) and \(C_i^{(1)}\cup C_i^{(2)}\). The resulting graph is \(\mathbb{R}\)-cospectral with \(\Gamma\).
\end{theorem}

If \(t=1\) and \(|C_1^{(1)}|=|C_1^{(2)}|=2\), then WQH-switching is equivalent to GM-switching on \(C_1^{(1)}\cup C_1^{(2)}\). But in general, they are different operations.
Note that the level of the corresponding matrix is equal to the least common multiple of the numbers $|C_i|/2$, $1\leq i\leq t$, just like for GM-switching.

\section{Switching methods of level 2}\label{sec:level2}

In general, a switching method corresponds to a conjugation of the adjacency matrix with a matrix \(Q\) of the form \begin{equation}\label{eq:Q}
    Q=\begin{bmatrix}R&O\\O&I\end{bmatrix}
\end{equation} where \(R\) is an orthogonal matrix. In this paper, we are interested in the case when \(R\) is a regular orthogonal matrix of level 2, since Wang and Xu \cite{level2mats} suggest that most \(\mathbb{R}\)-cospectral graphs can be constructed using such matrices. In the following, we will refer to them simply as \emph{switching methods of level 2}. We also restrict to the case when \(R\) is indecomposable, since those switching methods have just one switching set and are therefore simpler to describe (and more useful in applications). Moreover, intuitively, switching methods on multiple switching sets may often be obtained by performing the switching operation in multiple steps, see also Section~\ref{sec:reducibility}.

These level $2$ switching methods were introduced by Abiad and Haemers \cite{AHswitching} algebraically, and expanded on combinatorially by Mao, Wang, Liu and Qiu \cite[Section~\ref{sec:sun}]{MAO2023} and Mao and Yan \cite{MaoSimilar}. Our approach to obtain the main results is similar to theirs, and will be described below.

\subsection{The approach}\label{sec:approach}

Our starting point is the following classification of indecomposable regular orthogonal matrices of level \(2\), which follows from the classification of weighing matrices of weight \(4\) by Chan, Rodger and Seberry \cite{weighing} and has been restated by Wang and Xu \cite
{level2mats} in the form below. 

\begin{theorem}[\cite{weighing}]\label{thm:level2mats}
Up to a permutation of the rows and columns, an indecomposable regular orthogonal matrix of level \(2\) and row sum \(1\) is one of the following:
\[{\textstyle
(i)\ \frac{1}{2}\left[
\arraycolsep=4pt\def\arraystretch{1}\begin{array}{rrrr}
-1 & 1 &  1 &  1 \\
1 & -1 &  1 &  1 \\
1 &  1 & -1 &  1 \\
1 &  1 &  1 & -1
\end{array}
\right]\, ,
\ (ii)\ \frac{1}{2}\left[
\arraycolsep=2.5pt\def\arraystretch{.9}\begin{array}{cccccc}
J       & O      & \cdots & \cdots & O      & Y      \\
Y       & J      & O      & \cdots & \cdots & O      \\
O       & Y      & J      & O      & \cdots & O      \\
        & \ddots & \ddots & \ddots & \ddots &        \\
O       & \cdots & O      & Y      & J      & O      \\
O       & \cdots & \cdots & O      & Y      & J
\end{array}
\right],
}\]

\[{\textstyle
\ (iii)\ \frac{1}{2}\left[
\arraycolsep=4pt\def\arraystretch{.9}\begin{array}{rrrrrrr}
-1 & 1 & 1 & 0 & 1 & 0 & 0 \\
0 & -1 & 1 & 1 & 0 & 1 & 0 \\
0 & 0 & -1 & 1 & 1 & 0 & 1 \\
1 & 0 & 0 & -1 & 1 & 1 & 0 \\
0 & 1 & 0 & 0 & -1 & 1 & 1 \\
1 & 0 & 1 & 0 & 0 & -1 & 1 \\
1 & 1 & 0 & 1 & 0 & 0 & -1
\end{array}
\right],
\ (iv)\
\frac{1}{2}\left[\arraycolsep=3pt\def\arraystretch{.9}\begin{array}{rrrr}
-I & I & I & I \\
I & -Z & I & Z \\
I & Z & -Z & I \\
I & I & Z & -Z
\end{array}\right],
}
\]
where $I$, $J$, $O$, $Y=2I-J$ and $Z=J-I$, are square matrices of order $2$.
\end{theorem}

The matrix in Theorem~\ref{thm:level2mats}(i) corresponds to GM-switching with respect to a switching set of size 4 (or equivalently, WQH-switching with one switching set of size 4). It can be seen as a special case of the infinite family of matrices of even order in Theorem~\ref{thm:level2mats}(ii). Notice that there are two different matrices of order 8 in Theorem~\ref{thm:level2mats}. One of them is part of the infinite family in Theorem~\ref{thm:level2mats}(ii). The other one is the matrix in Theorem~\ref{thm:level2mats}(iv). Since it is not part of the infinite family, we call it the sporadic matrix of order 8.

Consider any switching operation of level 2. Algebraically, this corresponds to a conjugation of the adjacency matrix with the matrix \(Q\) of equation~\eqref{eq:Q}, where \(R\) has level 2. Let \begin{equation}\label{eq:A}
    A=\begin{bmatrix}B&V\\V^T&W\end{bmatrix}
\end{equation}
be the adjacency matrix of a graph, labelled accordingly. Our goal is to see when the switching method works. That is, we want to find conditions on \(B\) and \(V\) such that \[A':=Q^TAQ=\begin{bmatrix}R^TBR&R^TV\\V^TR&W\end{bmatrix}\] is again an adjacency matrix. Thus, given an indecomposable regular orthogonal matrix \(R\) of level 2 from Theorem~\ref{thm:level2mats}, we first determine the possible columns of \(V\), in other words, the column vectors \(v\) such that \(R^Tv\) is again a 01-vector, and then find the adjacency matrices \(B\) such that \(R^TBR\) is also an adjacency matrix:

\begin{definition}\label{def:Bmatrices}
    Given a regular orthogonal matrix \(R\), define \(\mathcal{V}_R\) as the set of 01-vectors \(v\) such that \(R^Tv\) is again a 01-vector, and define \(\mathcal{B}_R\) as the set of adjacency matrices \(B\) for which \(R^TBR\) is again an adjacency matrix.
\end{definition}

In order to fully describe a switching operation, one needs to figure out what the sets $\mathcal{V}_R$ and $\mathcal{B}_R$ are.

From now on, all switching methods described in this paper are coming from a conjugation with one of the matrices in Theorem~\ref{thm:level2mats}. In particular, these methods of level 2 can be used to construct \(\mathbb{R}\)-cospectral graphs (see the characterization of \(\mathbb{R}\)-cospectral graphs by Johnson and Newman \cite{rcospectral}).

\subsection{Abiad-Haemers switching}\label{sec:AHswitching}

The regular orthogonal matrices of level 2 with order 6, 7 and 8 were studied by Abiad and Haemers \cite{AHswitching} in an algebraic way.
We refer to them as \emph{Six, Seven and Eight vertex AH-switching}, respectively. For Six vertex AH-switching and Seven vertex AH-switching, Abiad and Haemers worked out exact requirements for the switching method to work. For Eight vertex AH-switching, only a partial answer was given.

In Section~\ref{sec:AH6}, Section~\ref{sec:fano} and Section~\ref{sec:cube} respectively, we restate these switching methods combinatorially. Due to this combinatorial interpretation, we nickname Seven vertex AH-switching as \emph{Fano switching} and Eight vertex AH-switching as \emph{Cube switching}. We also give an application of Fano-switching.

\subsection{Mao-Wang-Liu-Qiu switching or Sun graph switching}\label{sec:sun}

In 2023, Mao, Wang, Liu and Qiu \cite{MAO2023} described a switching method of level 2 that corresponds to a conjugation of the adjacency matrix with the matrix in Theorem~\ref{thm:level2mats}(ii). The switching set is a so-called \emph{generalized sun graph}. For this reason, we call the operation \emph{Sun graph switching}.

The formulation of Theorem~\ref{thm:sun} below is a reformulation of the result of Mao, Wang, Liu and Qiu \cite{MAO2023}. We restrict to the case when \(m\) is odd, since otherwise the matrix in their proof is decomposable, see \cite[Remark~3]{MAO2023}. To understand the equivalence between the original formulation and the one in this paper, label the sets \(C_i\), \(i\in\mathbb{Z}/2m\mathbb{Z}\), as \(C_i=\{v_{4i-1},v_{4i}\}\).
In the original result, there is only one option for \(C_i\cup C_j\) when \(j=i+\frac{m-1}{2}\). However, the two formulations are equivalent, because the switching sets are isomorphic via a permutation that does not change the conditions on the outside vertices.

\begin{theorem}[Sun graph switching \cite{MAO2023}]\label{thm:sun}
    Let \(\Gamma\) be a graph with a vertex partition \(\{C_1,\dots,\\C_m,D\}\), \(m\) odd, \(m\geq3\), such that:
    \begin{enumerate}[(i)]
        \item \(|C_1|=\cdots=|C_m|=2\).
        \item Every vertex in \(D\) has the same number of neighbours in \(C_1,\dots,C_m\) modulo \(2\).
        \item For all \(i,j\in\mathbb{Z}/m\mathbb{Z}\) with \(j\in\{i+1,\dots,i+\frac{m-1}{2}\}\), the induced subgraph on \(C_i\cup C_j\) is
        
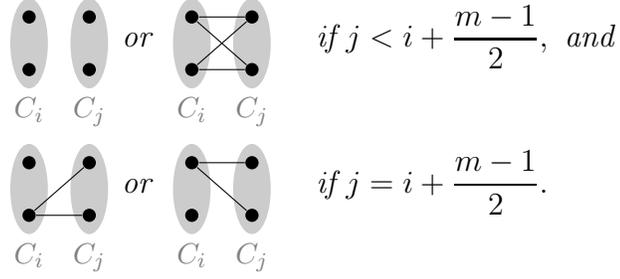
\begin{align*}
        \begin{tikzpicture}[baseline=(O.base)]
    \path (0,.3) node (O) {};
    \path (-.4,0) rectangle (1,0);
    \fill[ss] (0,.35) ellipse (.25 and .6)
    (.8,.35) ellipse (.25 and .6);
    \path[every node/.append style={circle, fill=black, minimum size=5pt, label distance=2pt, inner sep=0pt}]
    (0,0) node[label={[label distance=5pt]270:\color{sslabel}\small\(C_i\)}] (1) {}
    (0,.7) node (2) {}
    (.8,0) node[label={[label distance=\labeljdist]270:\color{sslabel}\small\(C_{\labelj}\)}] (3) {}
    (.8,.7) node (4) {};
\end{tikzpicture}\text{ or }\begin{tikzpicture}[baseline=(O.base)]
    \path (0,.3) node (O) {};
    \path (-.4,0) rectangle (1,0);
    \fill[ss] (0,.35) ellipse (.25 and .6)
    (.8,.35) ellipse (.25 and .6);
    \path[every node/.append style={circle, fill=black, minimum size=5pt, label distance=2pt, inner sep=0pt}]
    (0,0) node[label={[label distance=5pt]270:\color{sslabel}\small\(C_i\)}] (1) {}
    (0,.7) node (2) {}
    (.8,0) node[label={[label distance=\labeljdist]270:\color{sslabel}\small\(C_{\labelj}\)}] (3) {}
    (.8,.7) node (4) {};
    \draw (1) edge (3) edge (4)
    (2) edge (3) edge (4);
\end{tikzpicture}&\quad\text{ if }j<i+\frac{m-1}{2},\text{ and }\\
        \begin{tikzpicture}[baseline=(O.base)]
    \path (0,.3) node (O) {};
    \path (-.4,0) rectangle (1,0);
    \fill[ss] (0,.35) ellipse (.25 and .6)
    (.8,.35) ellipse (.25 and .6);
    \path[every node/.append style={circle, fill=black, minimum size=5pt, label distance=2pt, inner sep=0pt}]
    (0,0) node[label={[label distance=5pt]270:\color{sslabel}\small\(C_i\)}] (1) {}
    (0,.7) node (2) {}
    (.8,0) node[label={[label distance=\labeljdist]270:\color{sslabel}\small\(C_{\labelj}\)}] (3) {}
    (.8,.7) node (4) {};
    \draw (1) edge (3) edge (4);
\end{tikzpicture}\text{ or }\begin{tikzpicture}[baseline=(O.base)]
    \path (0,.3) node (O) {};
    \path (-.4,0) rectangle (1,0);
    \fill[ss] (0,.35) ellipse (.25 and .6)
    (.8,.35) ellipse (.25 and .6);
    \path[every node/.append style={circle, fill=black, minimum size=5pt, label distance=2pt, inner sep=0pt}]
    (0,0) node[label={[label distance=5pt]270:\color{sslabel}\small\(C_i\)}] (1) {}
    (0,.7) node (2) {}
    (.8,0) node[label={[label distance=\labeljdist]270:\color{sslabel}\small\(C_{\labelj}\)}] (3) {}
    (.8,.7) node (4) {};
    \draw (2) edge (3) edge (4);
\end{tikzpicture}&\quad\text{ if }j=i+\frac{m-1}{2}.
        \end{align*}
      \end{enumerate}
    Let \(\pi\) be the permutation on \(C_1\cup\dots\cup C_m\) that shifts the vertices cyclically to the left. In other words, if we denote \(C_i=\{v_{i1},v_{i2}\}\) for all \(i\), then \(\pi(v_{i1})=v_{i-1,1}\) and \(\pi(v_{i2})=v_{i-1,2}\). For every \(v\in D\) that has exactly one neighbour \(w\) in each \(C_i\), replace each edge \(\{v,w\}\) by \(\{v,\pi(w)\}\). For all \(i\in\mathbb{Z}/m\mathbb{Z}\), replace the induced subgraph on \(C_i\cup C_{i+\frac{m-1}{2}}\) by the former induced subgraph on \(C_i\cup C_{i+\frac{m+1}{2}}\). The resulting graph is \(\mathbb{R}\)-cospectral with \(\Gamma\).
\end{theorem}

Note that the induced graphs of order two on the \(C_i\)'s are edgeless. By considering the complement, Theorem \ref{thm:sun} also covers the case where each $C_i$ is the complete graph \(K_2\).

\section{Two new switching methods of level 2}\label{sec:new}

In this section, we describe two switching operations coming from the regular orthogonal \(2m\times2m\) matrix \(R_{2m}=\frac12\:\mathrm{circulant}(J,O,\dots,O,Y)\)
from Theorem~\ref{thm:level2mats}(ii).
Following the notation of Section~\ref{sec:approach}, the possible columns of \(V\) in equation~\eqref{eq:A} are determined by the following result, which is a generalization of \cite[Lemma~5]{AHswitching}. 

\begin{lemma}[{\cite[Lemma~15]{phdaida}}]
\label{lem:phdaida}
    \(v\in\mathcal{V}_R\) if and only if the numbers of ones in each of the tuples \((v_1,v_2),\dots,(v_{2m-1},v_{2m})\), is the same modulo 2.
    If the number of ones is always even (0 or 2), then \(R_{2m}^Tv=v\). If it is always equal to 1, then \(R_{2m}^Tv=(v_3,\dots,v_{2m},v_1,v_2)^T\).
\end{lemma}

A proof of Lemma~\ref{lem:phdaida} was also given later by Mao, Wang, Liu and Qiu \cite[Theorem~6.1]{MAO2023}.

In order to investigate what the matrix \(B\) looks like, we mimic part of the proof of \cite[Lemma~6]{AHswitching}, using similar notation.
Denote \[B=\begin{bmatrix}B_{11}&\cdots&B_{12}\\\vdots&\ddots&\vdots\\B_{m1}&\cdots&B_{mm}\end{bmatrix}\quad\text{ and } \quad B':=R_{2m}^TBR_{2m}=\begin{bmatrix}B'_{11}&\cdots&B'_{12}\\\vdots&\ddots&\vdots\\B'_{m1}&\cdots&B'_{mm}\end{bmatrix}\] where the \(B_{ij}\) and \(B'_{ij}\) blocks are \(2\times2\) matrices.
By definition of \(B'\), we have
\begin{equation}\label{eq:1}
    4B'_{ij}=\underbrace{JB_{ij}J}_{aJ}+\underbrace{JB_{i,j+1}Y}_{b\left[\arraycolsep=2.5pt\def\arraystretch{.9}\begin{array}{rr}1&-1\\1&-1\end{array}\right]}+\underbrace{YB_{i+1,j}J}_{c\left[\arraycolsep=1pt\def\arraystretch{.9}\begin{array}{rr}1&1\\-1&-1\end{array}\right]}+\underbrace{YB_{i+1,j+1}Y}_{dY}
\end{equation}
for certain integers \(a,b,c,d\).
This equation, for all \(i,j\), is in fact \emph{equivalent} to the definition of \(B'\).
In a similar way, by writing \(B=R_{2m}B'R_{2m}^T\), we obtain
\begin{equation}\label{eq:2}
    4B_{ij}=\underbrace{JB'_{ij}J}_{aJ}+\underbrace{JB'_{i,j-1}Y}_{b\left[\arraycolsep=2.5pt\def\arraystretch{.9}\begin{array}{rr}1&-1\\1&-1\end{array}\right]}+\underbrace{YB'_{i-1,j}J}_{c\left[\arraycolsep=1pt\def\arraystretch{.9}\begin{array}{rr}1&1\\-1&-1\end{array}\right]}+\underbrace{YB'_{i-1,j-1}Y}_{dY}
\end{equation}
for certain integers \(a,b,c,d\).

Note that \(J,\left[\arraycolsep=2.5pt\def\arraystretch{.9}\begin{array}{rr}1&-1\\1&-1\end{array}\right],\left[\arraycolsep=1pt\def\arraystretch{.9}\begin{array}{rr}1&1\\-1&-1\end{array}\right]\) and \(Y\) are linearly independent. So, for example, if \(B_{ij}=O\), then in equation~\eqref{eq:2}, \(a=b=c=d=0\) and hence \(B'_{ij}=O\), \(B'_{i,j-1}\) has constant column sum, \(B'_{i-1,j}\) has constant row sum and \(B'_{i-1,j-1}\) has constant diagonal sum.


\begin{theorem}\label{thm:infinite1}
    Let \(\Gamma\) be a graph with a vertex partition \(\{C_1,\dots,C_m,D\}\), \(m\) odd, \(m\geq3\), such that:
    \begin{enumerate}[(i)]
        \item \(|C_1|=\cdots=|C_m|=2\).
        \item Every vertex in \(D\) has the same number of neighbours in \(C_1,\dots,C_m\) modulo \(2\).
        \item For all \(i,j\in\mathbb{Z}/m\mathbb{Z}\) with \(j\in\{i+1,\dots,i+\frac{m-1}{2}\}\), the induced subgraph on \(C_i\cup C_j\) is
        
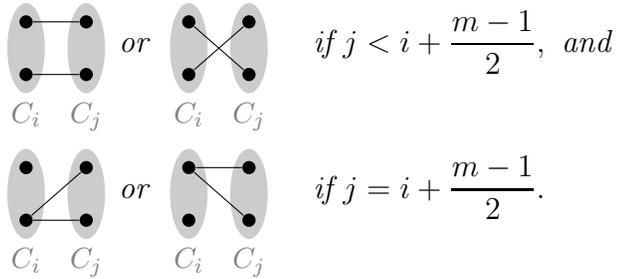
\begin{align*}
        \begin{tikzpicture}[baseline=(O.base)]
    \path (0,.3) node (O) {};
    \path (-.4,0) rectangle (1,0);
    \fill[ss] (0,.35) ellipse (.25 and .6)
    (.8,.35) ellipse (.25 and .6);
    \path[every node/.append style={circle, fill=black, minimum size=5pt, label distance=2pt, inner sep=0pt}]
    (0,0) node[label={[label distance=5pt]270:\color{sslabel}\small\(C_i\)}] (1) {}
    (0,.7) node (2) {}
    (.8,0) node[label={[label distance=\labeljdist]270:\color{sslabel}\small\(C_{\labelj}\)}] (3) {}
    (.8,.7) node (4) {};
    \draw (1) edge (3)
    (2) edge (4);
\end{tikzpicture}\text{ or }\begin{tikzpicture}[baseline=(O.base)]
    \path (0,.3) node (O) {};
    \path (-.4,0) rectangle (1,0);
    \fill[ss] (0,.35) ellipse (.25 and .6)
    (.8,.35) ellipse (.25 and .6);
    \path[every node/.append style={circle, fill=black, minimum size=5pt, label distance=2pt, inner sep=0pt}]
    (0,0) node[label={[label distance=5pt]270:\color{sslabel}\small\(C_i\)}] (1) {}
    (0,.7) node (2) {}
    (.8,0) node[label={[label distance=\labeljdist]270:\color{sslabel}\small\(C_{\labelj}\)}] (3) {}
    (.8,.7) node (4) {};
    \draw (1) edge (4)
    (2) edge (3);
\end{tikzpicture}&\quad\text{ if }j<i+\frac{m-1}{2},\text{ and }\\
        \begin{tikzpicture}[baseline=(O.base)]
    \path (0,.3) node (O) {};
    \path (-.4,0) rectangle (1,0);
    \fill[ss] (0,.35) ellipse (.25 and .6)
    (.8,.35) ellipse (.25 and .6);
    \path[every node/.append style={circle, fill=black, minimum size=5pt, label distance=2pt, inner sep=0pt}]
    (0,0) node[label={[label distance=5pt]270:\color{sslabel}\small\(C_i\)}] (1) {}
    (0,.7) node (2) {}
    (.8,0) node[label={[label distance=\labeljdist]270:\color{sslabel}\small\(C_{\labelj}\)}] (3) {}
    (.8,.7) node (4) {};
    \draw (1) edge (3) edge (4);
\end{tikzpicture}\text{ or }\begin{tikzpicture}[baseline=(O.base)]
    \path (0,.3) node (O) {};
    \path (-.4,0) rectangle (1,0);
    \fill[ss] (0,.35) ellipse (.25 and .6)
    (.8,.35) ellipse (.25 and .6);
    \path[every node/.append style={circle, fill=black, minimum size=5pt, label distance=2pt, inner sep=0pt}]
    (0,0) node[label={[label distance=5pt]270:\color{sslabel}\small\(C_i\)}] (1) {}
    (0,.7) node (2) {}
    (.8,0) node[label={[label distance=\labeljdist]270:\color{sslabel}\small\(C_{\labelj}\)}] (3) {}
    (.8,.7) node (4) {};
    \draw (2) edge (3) edge (4);
\end{tikzpicture}&\quad\text{ if }j=i+\frac{m-1}{2}.
        \end{align*}
    \end{enumerate}
    Let \(\pi\) be the permutation on \(C_1\cup\dots\cup C_m\) that shifts the vertices cyclically to the left. For every \(v\in D\) that has exactly one neighbour \(w\) in each \(C_i\), replace each edge \(\{v,w\}\) by \(\{v,\pi(w)\}\). For all \(i,j\in\mathbb{Z}/m\mathbb{Z}\) with \(j\in\{i+1,\dots,i+\frac{m-1}{2}\}\), replace the induced subgraph on \(C_i\cup C_j\) by the former induced subgraph on \(C_{i+1}\cup C_{j+1}\) if \(j<i+\frac{m-1}{2}\) and by the former induced subgraph on \(C_{i+1}\cup C_j\) if \(j=i+\frac{m-1}{2}\). The resulting graph is \(\mathbb{R}\)-cospectral with \(\Gamma\).
\end{theorem}
\begin{proof}
Following the notation of Section~\ref{sec:approach}, we label the vertices of the graph in such a way that its adjacency matrix and the adjacency matrix of the new graph have block form
    \[A=\begin{bmatrix}B&V\\V^T&W\end{bmatrix}\text{ and }A'=\begin{bmatrix}B'&V'\\{V'}^T&W\end{bmatrix}\] respectively, where \(B\) and \(B'\) are the adjacency matrices of the induced subgraph on \(C_1\cup\dots\cup C_m\). Let \(N=\begin{bmatrix}0&0\\1&1\end{bmatrix}\). Then condition (iii) says that  
    \[B=\mathrm{circulant}(O,\underbrace{I,\dots,I}_{(m-3)/2},N,N^T,\underbrace{I,\dots,I}_{(m-3)/2})\] where any \(2\times2\) block \(B_{ij}\) (together with its transpose \(B_{ji}\)), \(i\neq j\), might be replaced by \(J-B_{ij}\) (resp.\ \(J-B_{ji}\)).
    We check that \(A'=Q^TAQ\), where \(Q=\begin{bmatrix}R_{2m}&O\\O&I\end{bmatrix}\).
    By Lemma~\ref{lem:phdaida}, \(V'=R_{2m}^TV\) is satisfied.
    We use equation~\eqref{eq:1} to verify that \(B'=R_{2m}^TBR_{2m}\). If \(j\in\{i,\dots,i+\frac{m-3}{2}\}\), then \(B_{ij}'=B_{i+1,j+1}\). If \(j=i+\frac{m-1}{2}\), then \(B_{ij}'=B_{i+1,j}\). This corresponds to the switching operation in the statement. So the adjacency matrices \(A\) and \(A'\) are similar with respect to the regular orthogonal matrix \(Q\), and therefore \(\R\)-cospectral.
\end{proof}

Next, we state a second new switching method.

\begin{theorem}\label{thm:infinite2}
    Let \(\Gamma\) be a graph with a vertex partition \(\{C_1,\dots,C_m,D\}\), \(m\geq5\), such that:
    \begin{enumerate}[(i)]
        \item \(|C_1|=\cdots=|C_m|=2\).
        \item Every vertex in \(D\) has the same number of neighbours in \(C_1,\dots,C_m\) modulo \(2\).
        \item For all \(i,j\in\mathbb{Z}/m\mathbb{Z}\) with \(j\in\{i+1,\dots,i+\left\lfloor\frac{m}{2}\right\rfloor\}\), the induced subgraph on \(C_i\cup C_j\) is
        
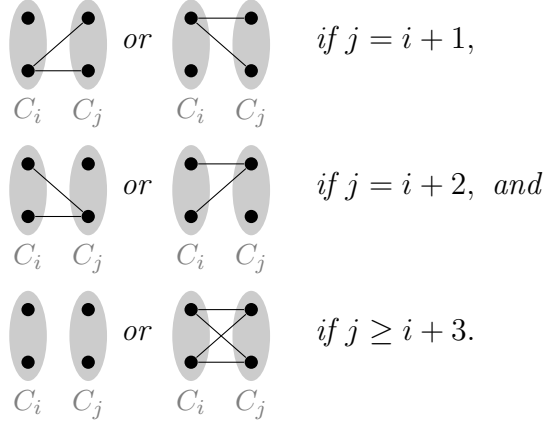
\begin{align*}
        \begin{tikzpicture}[baseline=(O.base)]
    \path (0,.3) node (O) {};
    \path (-.4,0) rectangle (1,0);
    \fill[ss] (0,.35) ellipse (.25 and .6)
    (.8,.35) ellipse (.25 and .6);
    \path[every node/.append style={circle, fill=black, minimum size=5pt, label distance=2pt, inner sep=0pt}]
    (0,0) node[label={[label distance=5pt]270:\color{sslabel}\small\(C_i\)}] (1) {}
    (0,.7) node (2) {}
    (.8,0) node[label={[label distance=\labeljdist]270:\color{sslabel}\small\(C_{\labelj}\)}] (3) {}
    (.8,.7) node (4) {};
    \draw (1) edge (3) edge (4);
\end{tikzpicture}\text{ or }\begin{tikzpicture}[baseline=(O.base)]
    \path (0,.3) node (O) {};
    \path (-.4,0) rectangle (1,0);
    \fill[ss] (0,.35) ellipse (.25 and .6)
    (.8,.35) ellipse (.25 and .6);
    \path[every node/.append style={circle, fill=black, minimum size=5pt, label distance=2pt, inner sep=0pt}]
    (0,0) node[label={[label distance=5pt]270:\color{sslabel}\small\(C_i\)}] (1) {}
    (0,.7) node (2) {}
    (.8,0) node[label={[label distance=\labeljdist]270:\color{sslabel}\small\(C_{\labelj}\)}] (3) {}
    (.8,.7) node (4) {};
    \draw (2) edge (3) edge (4);
\end{tikzpicture}&\quad\text{ if }j=i+1,\\
        \begin{tikzpicture}[baseline=(O.base)]
    \path (0,.3) node (O) {};
    \path (-.4,0) rectangle (1,0);
    \fill[ss] (0,.35) ellipse (.25 and .6)
    (.8,.35) ellipse (.25 and .6);
    \path[every node/.append style={circle, fill=black, minimum size=5pt, label distance=2pt, inner sep=0pt}]
    (0,0) node[label={[label distance=5pt]270:\color{sslabel}\small\(C_i\)}] (1) {}
    (0,.7) node (2) {}
    (.8,0) node[label={[label distance=\labeljdist]270:\color{sslabel}\small\(C_{\labelj}\)}] (3) {}
    (.8,.7) node (4) {};
    \draw (1) edge (3)
    (2) edge (3);
\end{tikzpicture}\text{ or }\begin{tikzpicture}[baseline=(O.base)]
    \path (0,.3) node (O) {};
    \path (-.4,0) rectangle (1,0);
    \fill[ss] (0,.35) ellipse (.25 and .6)
    (.8,.35) ellipse (.25 and .6);
    \path[every node/.append style={circle, fill=black, minimum size=5pt, label distance=2pt, inner sep=0pt}]
    (0,0) node[label={[label distance=5pt]270:\color{sslabel}\small\(C_i\)}] (1) {}
    (0,.7) node (2) {}
    (.8,0) node[label={[label distance=\labeljdist]270:\color{sslabel}\small\(C_{\labelj}\)}] (3) {}
    (.8,.7) node (4) {};
    \draw (1) edge (4)
    (2) edge (4);
\end{tikzpicture}&\quad\text{ if }j=i+2,\text{ and }\\
        \begin{tikzpicture}[baseline=(O.base)]
    \path (0,.3) node (O) {};
    \path (-.4,0) rectangle (1,0);
    \fill[ss] (0,.35) ellipse (.25 and .6)
    (.8,.35) ellipse (.25 and .6);
    \path[every node/.append style={circle, fill=black, minimum size=5pt, label distance=2pt, inner sep=0pt}]
    (0,0) node[label={[label distance=5pt]270:\color{sslabel}\small\(C_i\)}] (1) {}
    (0,.7) node (2) {}
    (.8,0) node[label={[label distance=\labeljdist]270:\color{sslabel}\small\(C_{\labelj}\)}] (3) {}
    (.8,.7) node (4) {};
\end{tikzpicture}\text{ or }\begin{tikzpicture}[baseline=(O.base)]
    \path (0,.3) node (O) {};
    \path (-.4,0) rectangle (1,0);
    \fill[ss] (0,.35) ellipse (.25 and .6)
    (.8,.35) ellipse (.25 and .6);
    \path[every node/.append style={circle, fill=black, minimum size=5pt, label distance=2pt, inner sep=0pt}]
    (0,0) node[label={[label distance=5pt]270:\color{sslabel}\small\(C_i\)}] (1) {}
    (0,.7) node (2) {}
    (.8,0) node[label={[label distance=\labeljdist]270:\color{sslabel}\small\(C_{\labelj}\)}] (3) {}
    (.8,.7) node (4) {};
    \draw (1) edge (3) edge (4)
    (2) edge (3) edge (4);
\end{tikzpicture}&\quad\text{ if }j\geq i+3.
        \end{align*}
    \end{enumerate}
    Let \(\pi\) be the permutation on \(C_1\cup\dots\cup C_m\) that shifts the vertices cyclically to the left. For every \(v\in D\) that has exactly one neighbour \(w\) in each \(C_i\), replace each edge \(\{v,w\}\) by \(\{v,\pi(w)\}\). For all \(i\in\mathbb{Z}/m\mathbb{Z}\), replace the induced subgraph on \(C_i\cup C_{i+1}\) by the former induced subgraph on \(C_i\cup C_{i+2}\) and replace the induced subgraph on \(C_i\cup C_{i+2}\) by the former induced subgraph on \(C_{i+1}\cup C_{i+2}\). The resulting graph is \(\mathbb{R}\)-cospectral with \(\Gamma\).
\end{theorem}
\begin{proof}
    The proof is the same as the proof of Theorem~\ref{thm:infinite1}, except that now, \[B=\mathrm{circulant}(O,N,N^T,\underbrace{O,\dots,O}_{m-5},N,N^T)\] where any \(2\times2\) block \(B_{ij}\) (together with its transpose \(B_{ji}\)), \(i\neq j\), might be replaced by \(J-B_{ij}\) (resp.\ \(J-B_{ji}\)).
    In order to show that \(B'=R_{2m}^TBR_{2m}\), we use equation~\eqref{eq:1}  to verify that for all \(i\in\mathbb{Z}/m\mathbb{Z}\), \(B_{ii}'=O\), \(B_{i,i+1}'=B_{i,i+2}\), \(B_{i,i+2}'=B_{i+1,i+2}\) and \(B_{ij}'=B_{ij}\) if \(j\in\{i+3,\dots,i+\left\lfloor\frac{m}{2}\right\rfloor\}\). This corresponds to the switching operation in the statement.
\end{proof}

For example, the switching set from Theorem~\ref{thm:infinite2} may look like

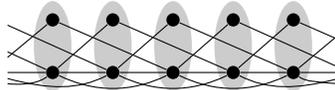
\begin{figure}[H]
    \centering
    \begin{tikzpicture}
\begin{scope}[yscale=1,xscale=-1]
    \fill[ss] (0,.35) ellipse (.25 and .6)
    (.8,.35) ellipse (.25 and .6)
    (1.6,.35) ellipse (.25 and .6)
    (2.4,.35) ellipse (.25 and .6)
    (3.2,.35) ellipse (.25 and .6);
    \path[every node/.append style={circle, fill=black, minimum size=5pt, label distance=2pt, inner sep=0pt}]
    (0,.7) node (1) {}
    (0,0) node (2) {}
    (.8,.7) node (3) {}
    (.8,0) node (4) {}
    (1.6,.7) node (5) {}
    (1.6,0) node (6) {}
    (2.4,.7) node (7) {}
    (2.4,0) node (8) {}
    (3.2,.7) node (9) {}
    (3.2,0) node (10) {};
    \draw (1) edge (-.6,.4375)
    (2) edge (4) edge (5) edge[out=-20,in=200] (6) edge (-.6,.5) edge (-.6,0) edge[out=200,in=5] (-.6,-.15)
    (3) edge (-.6,.07)
    (4) edge (1) edge (6) edge (7) edge[out=-20,in=200] (8) edge[out=200,in=-15] (-.6,-.08)
    (6) edge (3) edge (8) edge (9) edge[out=-20,in=200] (10)
    (8) edge (5) edge (10) edge (3.8,.63) edge[out=-20,in=195] (3.8,-.08)
    (9) edge (3.8,.2)
    (10) edge (7) edge (3.8,0) edge (3.8,.275) edge[out=-20,in=185] (3.8,-.15);
\end{scope}
\end{tikzpicture}
\end{figure}
before switching (the left- and rightmost vertices are connected cyclically). After the operation, it turns into:
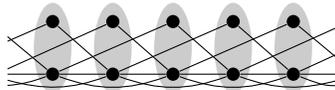
\begin{figure}[H]
    \centering
    \begin{tikzpicture}
    \fill[ss] (0,.35) ellipse (.25 and .6)
    (.8,.35) ellipse (.25 and .6)
    (1.6,.35) ellipse (.25 and .6)
    (2.4,.35) ellipse (.25 and .6)
    (3.2,.35) ellipse (.25 and .6);
    \path[every node/.append style={circle, fill=black, minimum size=5pt, label distance=2pt, inner sep=0pt}]
    (0,.7) node (1) {}
    (0,0) node (2) {}
    (.8,.7) node (3) {}
    (.8,0) node (4) {}
    (1.6,.7) node (5) {}
    (1.6,0) node (6) {}
    (2.4,.7) node (7) {}
    (2.4,0) node (8) {}
    (3.2,.7) node (9) {}
    (3.2,0) node (10) {};
    \draw (1) edge (-.6,.4375)
    (2) edge (4) edge (5) edge[out=-20,in=200] (6) edge (-.6,.5) edge (-.6,0) edge[out=200,in=5] (-.6,-.15)
    (3) edge (-.6,.07)
    (4) edge (1) edge (6) edge (7) edge[out=-20,in=200] (8) edge[out=200,in=-15] (-.6,-.08)
    (6) edge (3) edge (8) edge (9) edge[out=-20,in=200] (10)
    (8) edge (5) edge (10) edge (3.8,.63) edge[out=-20,in=195] (3.8,-.08)
    (9) edge (3.8,.2)
    (10) edge (7) edge (3.8,0) edge (3.8,.275) edge[out=-20,in=185] (3.8,-.15);
\end{tikzpicture}
\end{figure}

\begin{corollary}\label{cor:10vertex}
    Let \(\Gamma\) be a graph with a vertex partition \(\{C_1,\dots,C_5,D\}\) such that:
    \begin{enumerate}[(i)]
        \item \(|C_1|=\dots=|C_5|=2\).
        \item Every vertex in \(D\) has the same number of neighbours in \(C_1,\dots,C_5\) modulo \(2\).
        \item One of the following holds for all \(i\in\mathbb{Z}/5\mathbb{Z}\):
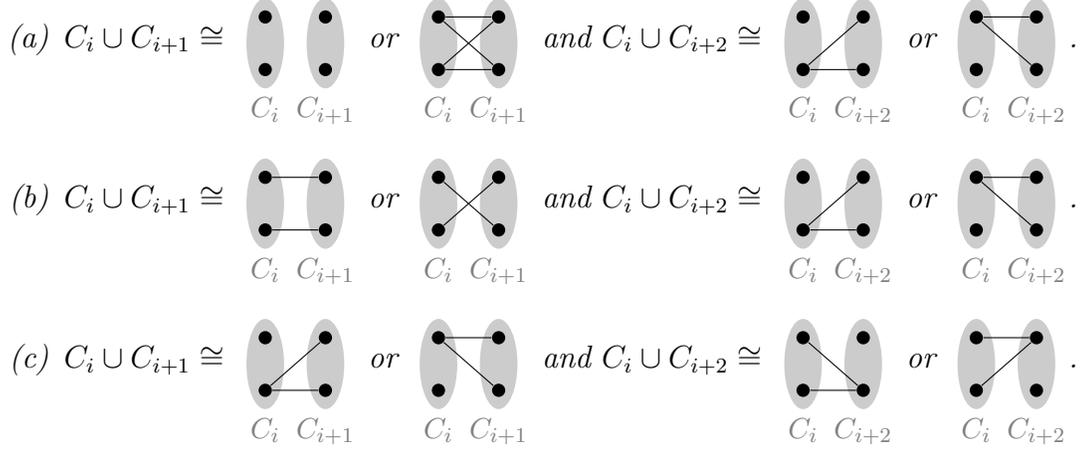
        \begin{enumerate}[(a)]
            \item \renewcommand{\labelj}{i+1}\renewcommand{\labeljdist}{.6pt}\(C_i\cup C_{i+1}\cong\) \begin{tikzpicture}[baseline=(O.base)]
    \path (0,.3) node (O) {};
    \path (-.4,0) rectangle (1,0);
    \fill[ss] (0,.35) ellipse (.25 and .6)
    (.8,.35) ellipse (.25 and .6);
    \path[every node/.append style={circle, fill=black, minimum size=5pt, label distance=2pt, inner sep=0pt}]
    (0,0) node[label={[label distance=5pt]270:\color{sslabel}\small\(C_i\)}] (1) {}
    (0,.7) node (2) {}
    (.8,0) node[label={[label distance=\labeljdist]270:\color{sslabel}\small\(C_{\labelj}\)}] (3) {}
    (.8,.7) node (4) {};
\end{tikzpicture} or \begin{tikzpicture}[baseline=(O.base)]
    \path (0,.3) node (O) {};
    \path (-.4,0) rectangle (1,0);
    \fill[ss] (0,.35) ellipse (.25 and .6)
    (.8,.35) ellipse (.25 and .6);
    \path[every node/.append style={circle, fill=black, minimum size=5pt, label distance=2pt, inner sep=0pt}]
    (0,0) node[label={[label distance=5pt]270:\color{sslabel}\small\(C_i\)}] (1) {}
    (0,.7) node (2) {}
    (.8,0) node[label={[label distance=\labeljdist]270:\color{sslabel}\small\(C_{\labelj}\)}] (3) {}
    (.8,.7) node (4) {};
    \draw (1) edge (3) edge (4)
    (2) edge (3) edge (4);
\end{tikzpicture} and \renewcommand{\labelj}{i+2}\(C_i\cup C_{i+2}\cong\) \begin{tikzpicture}[baseline=(O.base)]
    \path (0,.3) node (O) {};
    \path (-.4,0) rectangle (1,0);
    \fill[ss] (0,.35) ellipse (.25 and .6)
    (.8,.35) ellipse (.25 and .6);
    \path[every node/.append style={circle, fill=black, minimum size=5pt, label distance=2pt, inner sep=0pt}]
    (0,0) node[label={[label distance=5pt]270:\color{sslabel}\small\(C_i\)}] (1) {}
    (0,.7) node (2) {}
    (.8,0) node[label={[label distance=\labeljdist]270:\color{sslabel}\small\(C_{\labelj}\)}] (3) {}
    (.8,.7) node (4) {};
    \draw (1) edge (3) edge (4);
\end{tikzpicture} or \begin{tikzpicture}[baseline=(O.base)]
    \path (0,.3) node (O) {};
    \path (-.4,0) rectangle (1,0);
    \fill[ss] (0,.35) ellipse (.25 and .6)
    (.8,.35) ellipse (.25 and .6);
    \path[every node/.append style={circle, fill=black, minimum size=5pt, label distance=2pt, inner sep=0pt}]
    (0,0) node[label={[label distance=5pt]270:\color{sslabel}\small\(C_i\)}] (1) {}
    (0,.7) node (2) {}
    (.8,0) node[label={[label distance=\labeljdist]270:\color{sslabel}\small\(C_{\labelj}\)}] (3) {}
    (.8,.7) node (4) {};
    \draw (2) edge (3) edge (4);
\end{tikzpicture}.
            \item \renewcommand{\labelj}{i+1}\(C_i\cup C_{i+1}\cong\) \begin{tikzpicture}[baseline=(O.base)]
    \path (0,.3) node (O) {};
    \path (-.4,0) rectangle (1,0);
    \fill[ss] (0,.35) ellipse (.25 and .6)
    (.8,.35) ellipse (.25 and .6);
    \path[every node/.append style={circle, fill=black, minimum size=5pt, label distance=2pt, inner sep=0pt}]
    (0,0) node[label={[label distance=5pt]270:\color{sslabel}\small\(C_i\)}] (1) {}
    (0,.7) node (2) {}
    (.8,0) node[label={[label distance=\labeljdist]270:\color{sslabel}\small\(C_{\labelj}\)}] (3) {}
    (.8,.7) node (4) {};
    \draw (1) edge (3)
    (2) edge (4);
\end{tikzpicture} or \begin{tikzpicture}[baseline=(O.base)]
    \path (0,.3) node (O) {};
    \path (-.4,0) rectangle (1,0);
    \fill[ss] (0,.35) ellipse (.25 and .6)
    (.8,.35) ellipse (.25 and .6);
    \path[every node/.append style={circle, fill=black, minimum size=5pt, label distance=2pt, inner sep=0pt}]
    (0,0) node[label={[label distance=5pt]270:\color{sslabel}\small\(C_i\)}] (1) {}
    (0,.7) node (2) {}
    (.8,0) node[label={[label distance=\labeljdist]270:\color{sslabel}\small\(C_{\labelj}\)}] (3) {}
    (.8,.7) node (4) {};
    \draw (1) edge (4)
    (2) edge (3);
\end{tikzpicture} and \renewcommand{\labelj}{i+2}\(C_i\cup C_{i+2}\cong\) \begin{tikzpicture}[baseline=(O.base)]
    \path (0,.3) node (O) {};
    \path (-.4,0) rectangle (1,0);
    \fill[ss] (0,.35) ellipse (.25 and .6)
    (.8,.35) ellipse (.25 and .6);
    \path[every node/.append style={circle, fill=black, minimum size=5pt, label distance=2pt, inner sep=0pt}]
    (0,0) node[label={[label distance=5pt]270:\color{sslabel}\small\(C_i\)}] (1) {}
    (0,.7) node (2) {}
    (.8,0) node[label={[label distance=\labeljdist]270:\color{sslabel}\small\(C_{\labelj}\)}] (3) {}
    (.8,.7) node (4) {};
    \draw (1) edge (3) edge (4);
\end{tikzpicture} or \begin{tikzpicture}[baseline=(O.base)]
    \path (0,.3) node (O) {};
    \path (-.4,0) rectangle (1,0);
    \fill[ss] (0,.35) ellipse (.25 and .6)
    (.8,.35) ellipse (.25 and .6);
    \path[every node/.append style={circle, fill=black, minimum size=5pt, label distance=2pt, inner sep=0pt}]
    (0,0) node[label={[label distance=5pt]270:\color{sslabel}\small\(C_i\)}] (1) {}
    (0,.7) node (2) {}
    (.8,0) node[label={[label distance=\labeljdist]270:\color{sslabel}\small\(C_{\labelj}\)}] (3) {}
    (.8,.7) node (4) {};
    \draw (2) edge (3) edge (4);
\end{tikzpicture}.
            \item \renewcommand{\labelj}{i+1}\(C_i\cup C_{i+1}\cong\) \begin{tikzpicture}[baseline=(O.base)]
    \path (0,.3) node (O) {};
    \path (-.4,0) rectangle (1,0);
    \fill[ss] (0,.35) ellipse (.25 and .6)
    (.8,.35) ellipse (.25 and .6);
    \path[every node/.append style={circle, fill=black, minimum size=5pt, label distance=2pt, inner sep=0pt}]
    (0,0) node[label={[label distance=5pt]270:\color{sslabel}\small\(C_i\)}] (1) {}
    (0,.7) node (2) {}
    (.8,0) node[label={[label distance=\labeljdist]270:\color{sslabel}\small\(C_{\labelj}\)}] (3) {}
    (.8,.7) node (4) {};
    \draw (1) edge (3) edge (4);
\end{tikzpicture} or \begin{tikzpicture}[baseline=(O.base)]
    \path (0,.3) node (O) {};
    \path (-.4,0) rectangle (1,0);
    \fill[ss] (0,.35) ellipse (.25 and .6)
    (.8,.35) ellipse (.25 and .6);
    \path[every node/.append style={circle, fill=black, minimum size=5pt, label distance=2pt, inner sep=0pt}]
    (0,0) node[label={[label distance=5pt]270:\color{sslabel}\small\(C_i\)}] (1) {}
    (0,.7) node (2) {}
    (.8,0) node[label={[label distance=\labeljdist]270:\color{sslabel}\small\(C_{\labelj}\)}] (3) {}
    (.8,.7) node (4) {};
    \draw (2) edge (3) edge (4);
\end{tikzpicture} and \renewcommand{\labelj}{i+2}\(C_i\cup C_{i+2}\cong\) \begin{tikzpicture}[baseline=(O.base)]
    \path (0,.3) node (O) {};
    \path (-.4,0) rectangle (1,0);
    \fill[ss] (0,.35) ellipse (.25 and .6)
    (.8,.35) ellipse (.25 and .6);
    \path[every node/.append style={circle, fill=black, minimum size=5pt, label distance=2pt, inner sep=0pt}]
    (0,0) node[label={[label distance=5pt]270:\color{sslabel}\small\(C_i\)}] (1) {}
    (0,.7) node (2) {}
    (.8,0) node[label={[label distance=\labeljdist]270:\color{sslabel}\small\(C_{\labelj}\)}] (3) {}
    (.8,.7) node (4) {};
    \draw (1) edge (3)
    (2) edge (3);
\end{tikzpicture} or \begin{tikzpicture}[baseline=(O.base)]
    \path (0,.3) node (O) {};
    \path (-.4,0) rectangle (1,0);
    \fill[ss] (0,.35) ellipse (.25 and .6)
    (.8,.35) ellipse (.25 and .6);
    \path[every node/.append style={circle, fill=black, minimum size=5pt, label distance=2pt, inner sep=0pt}]
    (0,0) node[label={[label distance=5pt]270:\color{sslabel}\small\(C_i\)}] (1) {}
    (0,.7) node (2) {}
    (.8,0) node[label={[label distance=\labeljdist]270:\color{sslabel}\small\(C_{\labelj}\)}] (3) {}
    (.8,.7) node (4) {};
    \draw (1) edge (4)
    (2) edge (4);
\end{tikzpicture}.
        \end{enumerate}
    \end{enumerate}
    Let \(\pi\) be the permutation on \(C_1\cup\dots\cup C_5\) that shifts the vertices cyclically to the left. For every \(v\notin C\) that has exactly one neighbour \(w\) in each \(C_i\), replace each edge \(\{v,w\}\) by \(\{v,\pi(w)\}\). According to the cases above:
    \begin{enumerate}[(a)]
        \item For every \(i\in\mathbb{Z}/5\mathbb{Z}\), replace the induced subgraph on \(C_i\cup C_{i+2}\) by the former induced subgraph on \(C_i\cup C_{i+3}\).
        \item For every \(i\in\mathbb{Z}/5\mathbb{Z}\), replace the induced subgraph on \(C_i\cup C_{i+1}\) by the former induced subgraph on \(C_{i+1}\cup C_{i+2}\) and replace the induced subgraph on \(C_i\cup C_{i+2}\) by the former induced subgraph on \(C_i\cup C_{i+3}\).
        \item For every \(i\in\mathbb{Z}/5\mathbb{Z}\), replace the induced subgraph on \(C_i\cup C_{i+1}\) by the former induced subgraph on \(C_i\cup C_{i+2}\) and replace the induced subgraph on \(C_i\cup C_{i+2}\) by the former induced subgraph on \(C_{i+1}\cup C_{i+2}\).
    \end{enumerate} The resulting graph is \(\mathbb{R}\)-cospectral with \(\Gamma\).
\end{corollary}
\begin{proof}
    This is a special case of Theorem~\ref{thm:sun}, Theorem~\ref{thm:infinite1} and Theorem~\ref{thm:infinite2} when \(m=5\).
\end{proof}

We gave two infinite families of matrices of \(\mathcal{B}_{R_{2m}}\). Together with Sun graph switching, that makes three families. If \(m=5\), they form the three cases of Corollary~\ref{cor:10vertex}, which represents the complete set for \(m=5\) (see Theorem~\ref{thm:AH10}).
But as \(m\) increases, there are more adjacency matrices in \(\mathcal{B}_{R_{2m}}\). Mao and Yan \cite{MaoSimilar} describe similar but distinct families of switching methods that come from the regular orthogonal matrix \(R_{2m}\).
In this paper, we have chosen switching methods that only have one indecomposable block, and that we suspect cannot be written as a combination of previously known switching methods. We investigate this phenomenon in the next section.

\section{Reducibility}\label{sec:reducibility}

It was noticed by Abiad and Haemers \cite[end of Section~3]{AHswitching} that sometimes, Six vertex AH-switching can be obtained by GM-switching twice. In other words, certain switching operations might not provide new examples of cospectral graphs, because those are already obtainable by already known switching methods. Therefore, our aim is not just finding any new switching methods, but rather those that cannot be reduced to \emph{smaller switching methods}, that is, switching methods coming from regular orthogonal matrices of level 2 whose largest indecomposable block is smaller.
Below, we prove that this is indeed the case for Sun graph switching (Theorem~\ref{thm:sun}).
Also, in the classification of switching methods of level 2 up to size 12 in Section~\ref{sec:classification}, we restrict ourselves to those that cannot be obtained by smaller switching methods. We make this notion concrete in the following definition.

\renewcommand*{\thefootnote}{*}
\begin{definition}\label{def:reducibility}
    Consider a switching method of level 2, coming from a matrix \(R\) from Theorem~\ref{thm:level2mats}. Such a method is \emph{reducible} if it can be obtained by a sequence of \emph{smaller switching methods}, that is, switching methods of level 2, coming from matrices whose largest indecomposable block is smaller than the size of \(R\).
    Otherwise, it is called \emph{irreducible}. 
    
    More precisely, 
    an adjacency matrix \(B\in\mathcal{B}_R\) is called \emph{reducible with respect to \(R\)} if there are decomposable\footnote{In particular, \(R_i\) can be of the form \(R_i=\begin{bmatrix}R_i'&O\\O&I\end{bmatrix}\) for some smaller regular orthogonal matrix \(R_i'\).} regular orthogonal matrices $R_1,\ldots, R_s$ such that $R = R_1\cdots R_s$, and such that $(Q_1\cdots Q_i)^TA(Q_1\cdots Q_i)$ is an adjacency matrix for any adjacency matrix
    $$A=\begin{bmatrix}B&V\\V^T&W\end{bmatrix}$$ and each $i < s$, where $$Q_i=\begin{bmatrix}R_i&O\\O&I\end{bmatrix}.$$ If \(i=s\), this will automatically be fulfilled since \(B\in\mathcal{B}_R\). A switching method of level 2 is reducible if the adjacency matrix of the corresponding switching set is reducible.

\end{definition}

Although \(R\) is indecomposable, we allow the smaller switching methods in the above definition to correspond to matrices with multiple indecomposable blocks. For example, GM-switching with \(t=2\) and \(|C_1|=|C_2|=4\) (see Theorem~\ref{thm:GM}) has two such blocks of size 4. Hence, a switching method with a switching set of size 8 may be reducible by GM-switching on two sets, as illustrated by the following example:

\begin{example}\label{ex:reducible}
    Consider the following switching method of level 2, coming from a conjugation with \(R_8\), where the switching set has adjacency matrix \(B=\mathrm{circulant}(00100010)\).
    
    Let \(\Gamma\) be a graph with a vertex partition \(\{C_1,\dots,C_4,D\}\) such that \(|C_1|=\dots=|C_4|=2\), every vertex in \(D\) has the same number of neighbours in \(C_1,\dots,C_4\) modulo \(2\), and \(C_1\cup\dots\cup C_4\) is isomorphic to:
    
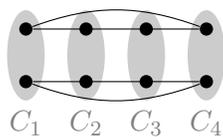
\[\begin{tikzpicture}[baseline=(O.base)]
    \path (0,.35) node (O) {};
    \fill[ss] (0,.35) ellipse (.25 and .6)
    (.8,.35) ellipse (.25 and .6)
    (1.6,.35) ellipse (.25 and .6)
    (2.4,.35) ellipse (.25 and .6);
    \path[every node/.append style={circle, fill=black, minimum size=5pt, label distance=2pt, inner sep=0pt}]
    (0,.7) node (1) {}
    (0,0) node[label={[label distance=\labeljdist]270:\color{sslabel}\small\(C_1\)}] (2) {}
    (.8,.7) node (3) {}
    (.8,0) node[label={[label distance=\labeljdist]270:\color{sslabel}\small\(C_2\)}] (4) {}
    (1.6,.7) node (5) {}
    (1.6,0) node[label={[label distance=\labeljdist]270:\color{sslabel}\small\(C_3\)}] (6) {}
    (2.4,.7) node (7) {}
    (2.4,0) node[label={[label distance=\labeljdist]270:\color{sslabel}\small\(C_4\)}] (8) {};
    \draw (2) edge (8) edge[out=-20,in=200] (8)
    (1) edge (7) edge[out=20,in=-200] (7);
\end{tikzpicture}\]
    Let \(\pi\) be the permutation on \(C_1\cup\dots\cup C_4\) that shifts the vertices cyclically to the left. For every \(v\notin C\) that has exactly one neighbour \(w\) in each \(C_i\), replace each edge \(\{v,w\}\) by \(\{v,\pi(w)\}\).
    One checks that the same operation can be obtained by two smaller switching methods; first, GM-switching with respect to the two sets \(C_1\cup C_2\) and \(C_3\cup C_4\), and then, once more, with respect to the set \(C_2\cup C_4\). Thus, the operation is reducible. Algebraically, this corresponds to the factorization
    \[R_8=\frac12\begin{bmatrix}J&O&O&Y\\Y&J&O&O\\O&Y&J&O\\O&O&Y&J\end{bmatrix}=\frac12\begin{bmatrix}J&Y&O&O\\Y&J&O&O\\O&O&J&Y\\O&O&Y&J\end{bmatrix}\cdot\frac12\begin{bmatrix}I&O&O&O\\O&J&O&Y\\O&O&I&O\\O&Y&O&J\end{bmatrix}\]
    and the fact that the adjacency condition on the outside vertices is met, and that \(B=\mathrm{circulant}(00100010)\) remains an adjacency matrix after conjugation with the first factor.
\end{example}

We start by proving that, if a switching method coming from the infinite family in Theorem~\ref{thm:level2mats}(ii) is reducible, then it is a product of matrices whose blocks are again part of this infinite family. Recall that \(R_{2m}\) is the regular orthogonal \(2m\times2m\) matrix \(R_{2m}=\frac12\:\mathrm{circulant}(J,O,\dots,O,Y)\)
from Theorem~\ref{thm:level2mats}(ii).

\begin{lemma}\label{lem:infreducible}
    If \(B\) is reducible with respect to \(R_{2m}\), then it must come from a sequence of switching methods corresponding to matrices of the form \(\mathrm{diag}(R_{2m_1},\dots,R_{2m_d})\) with \(m_1+\dots+m_d\leq m\). Moreover, if we denote its switching set by \(C_1\cup\dots\cup C_m\), with the property that every vertex outside the set, has the same number of neighbours in each \(C_i\) modulo 2 (see Lemma~\ref{lem:phdaida}), then the switching sets of the smaller methods are a union of \(C_i\)'s.
\end{lemma}
\begin{proof}
    Suppose, by contradiction, that at some point in the sequence of smaller methods, we can use a matrix where one of the indecomposable blocks corresponds to Seven vertex AH-switching (Fano switching) or Eight vertex AH-switching (Cube switching). Looking at the conditions on the outside vertices for both methods (see Lemma~\ref{lem:fano} and Lemma~\ref{lem:cube}), one can always choose an outside vertex that violates these conditions.

    If we can apply a smaller switching method from the infinite family in Theorem~\ref{thm:level2mats}(ii), then its switching set must be the union of a subset of these pairs (otherwise, we can always construct a vertex that harms the modulo 2 property). In particular, in every step, the outside vertices maintain the property that they have the same number of neighbours in \(C_1,\dots,C_m\) modulo 2 (using, again, Lemma~\ref{lem:phdaida}).
\end{proof}

Using the above lemma, we can show that Sun graph switching is irreducible.

\begin{theorem}\label{thm:sunirred}
    Sun graph switching (Theorem~\ref{thm:sun}) is irreducible for all \(m\), \(m\) odd, \(m\geq3\).
\end{theorem}
\begin{proof}
    Suppose, by contradiction, that Sun graph switching is reducible. Lemma~\ref{lem:infreducible} implies that it must be a product of smaller switching methods from the infinite family of Theorem~\ref{thm:level2mats}(ii). Choose any of those smaller switching methods and denote its switching set by \(S\). Then \(S\) is a union of \(C_i\)'s from the original switching set (see Lemma~\ref{lem:infreducible}). Let \(i\in\mathbb{Z}/m\mathbb{Z}\) such that \(C_i\subsetneq S\) (if \(S=\emptyset\) or \(S=C_i\) for some \(i\), the smaller switching method is trivial, or an automorphism interchanging two vertices).
    
    A vertex of \(C_{i+\frac{m-1}{2}}\) is adjacent with just one of the vertices of \(C_i\), but it is adjacent to an even number of vertices of any \(C_j\), \(i\neq j\). So \(C_{i+\frac{m-1}{2}}\subsetneq S\) as well (otherwise, Theorem~\ref{thm:sun}(ii) is violated). Repeating this argument once, proves that \(C_{i-1}\subsetneq S\). By induction, \(S=C_1\cup\dots\cup C_m\). We now prove that it has to be the original switching method again.
    Suppose that it is not. Then the smaller switching method is a switching method with multiple non-trivial indecomposable blocks (for example, GM-switching with \(t\geq2\) in Theorem~\ref{thm:GM}). In other words, the smaller switching method corresponds to a conjugation of \(B=\mathrm{circulant}(O,O,\dots,O,N,N^T,O,\dots,O)\) -- up to a permutation of pairs of rows and columns, and up to the replacement of a block by its complement -- with a block diagonal matrix \(\mathrm{diag}(R_{2m_1},\dots,R_{2m_k})\), that is,   
\begin{align*}
    \begin{bmatrix}R_{2m_1}&&\\&\hspace{-3mm}\ddots&\\&&\hspace{-3mm}R_{2m_k}\end{bmatrix}^T\begin{bmatrix}B^{11}&\hspace{-3mm}\cdots&\hspace{-3mm}B^{1k}\\\vdots&\hspace{-3mm}\ddots&\hspace{-3mm}\vdots\\B^{k1}&\hspace{-3mm}\cdots&\hspace{-3mm}B^{kk}\end{bmatrix}\begin{bmatrix}R_{2m_1}&&\\&\hspace{-3mm}\ddots&\\&&\hspace{-3mm}R_{2m_k}\end{bmatrix}
    =\begin{bmatrix}R_{2m_1}^TB^{11}R_{2m_1}&\hspace{-3mm}\cdots&\hspace{-3mm}R_{2m_1}^TB^{1k}R_{2m_k}\\\vdots&\hspace{-3mm}\ddots&\hspace{-3mm}\vdots\\R_{2m_k}^TB^{k1}R_{2m_1}&\hspace{-3mm}\cdots&\hspace{-3mm}R_{2m_k}^TB^{kk}R_{2m_k}\end{bmatrix}.
    \end{align*}
    Consider the indecomposable block of size \(2m_1\), and denote the corresponding subset by \(S_1\). Note that it is still a union of \(C_i\)'s. Suppose w.l.o.g.\ that \(C_1\subsetneq S_1\).
    
    If \(C_{1+\frac{m+1}{2}}\subseteq S_1\), then there is a block of \(B^{11}\) that is equal to \(N^T\) or \(J-N^T\). The \(4\times4\) block of \(B^{11}\) with that block in its upper left corner, is one of the matrices \[\left[\arraycolsep=1pt\def\arraystretch{.9}\begin{array}{lc}N^T&O\\O&O\end{array}\right],\left[\arraycolsep=1pt\def\arraystretch{.9}\begin{array}{lc}N^T&O\\O&N\end{array}\right],\left[\arraycolsep=1pt\def\arraystretch{.9}\begin{array}{lc}N^T&O\\O&N^T\end{array}\right],\left[\arraycolsep=1pt\def\arraystretch{.9}\begin{array}{lc}N^T&N\\O&O\end{array}\right],\left[\arraycolsep=1pt\def\arraystretch{.9}\begin{array}{lc}N^T&N\\O&N^T\end{array}\right],\left[\arraycolsep=1pt\def\arraystretch{.9}\begin{array}{lc}N^T&O\\N&O\end{array}\right],\left[\arraycolsep=1pt\def\arraystretch{.9}\begin{array}{lc}N^T&O\\N&N^T\end{array}\right],\left[\arraycolsep=1pt\def\arraystretch{.9}\begin{array}{lc}N^T&N\\N&O\end{array}\right]\] (up to replacing a \(2\times2\) block by its complement). By equation~\eqref{eq:1}, only those with \(N\) in the lower left corner can occur, 
    which means that \(C_2\subseteq S_1\).
    
    If \(C_{1+\frac{m+1}{2}}\not\subseteq S_1\), suppose w.l.o.g.\ that it is contained in \(S_2\). We can apply the same reasoning on the matrix \(R_{2m_1}B^{12}R_{2m_2}^T\) (noting that equation~\eqref{eq:1} is still valid if we conjugate with \(R_{2m}\)'s of different sizes) to conclude that \(C_2\subseteq S_1\).

    By induction, \(S_1=S\) and there is only one indecomposable block. That is, the switching method is the original one, which means that there is no proper factorization into smaller switching operations.
    \end{proof}

\section{Classification of irreducible switching methods of level 2 with one non-trivial indecomposable block, up to size 12}\label{sec:classification}

In this section, we classify all irreducible switching operations coming from a conjugation of the adjacency matrix with a matrix \(Q\) of the form \[Q=\begin{bmatrix}R&O\\O&I\end{bmatrix},\] where \(R\) is an indecomposable block of size at most 12. According to Theorem~\ref{thm:level2mats}, there are seven possibilities for \(R\) (up to permutation). We already noted that the matrix of size 4 in Theorem~\ref{thm:level2mats}(i) corresponds to GM-switching on one set of size 4.

\begin{remark}
    We should note that in the classification below, we do not include switching methods coming from matrices with multiple non-trivial indecomposable blocks, such as GM-switching when \(t\geq2\) (see Theorem~\ref{thm:GM}), or combinations of GM-switching and Six vertex AH-switching. Still, we defined reducibility (Definition~\ref{def:reducibility}) with those switching methods in mind. Therefore, this classification is not enough to understand all switching methods of level 2 up to size 12. The completion of this broader classification is posed as an open problem in Section~\ref{sec:concludingremarks}.
\end{remark}

To obtain the classification, we rely partially on the use of a computer. More information on our computer program and a link to the Sagemath code can be found in Appendix~\ref{app:code}.

Recall from Definition~\ref{def:Bmatrices} that \(\mathcal{B}_R\) is the set of adjacency matrices \(B\) for which \(R^TBR\) is again an adjacency matrix. It was already proved by Abiad and Haemers \cite{AHswitching} that \(\mathcal{B}_R\) is closed under complementation and under conjugation with a permutation matrix \(P\) for which \(P^TRP=R\). However, something slightly more general is true:

\begin{lemma}\label{lem:symmetry}
    \(\mathcal{B}_R\) is closed under complementation and under conjugation with a permutation matrix \(P\) 
    for which there exists a (possibly different) permutation matrix \(P'\) such that \(P^TRP'=R\).
\end{lemma}
\begin{proof}
    Let \(B\in\mathcal{B}_R\). Then \(R^TBR\) is an adjacency matrix.
    Since \(R\) is regular orthogonal, \(R^T(J-I-B)R=J-I-R^TBR\) is also an adjacency matrix (of the complement), hence \(J-I-B\in\mathcal{B}_R\).
    Suppose that \(P^TRP'=R\). Then \[R^T(P^TBP)R=(P'^TR^TP)(P^TBP)(P^TRP')=P'^T(R^TBR)P'\] is also an adjacency matrix (of the same graph, but permuted), hence \(P^TBP\in\mathcal{B}_R\).
\end{proof}

When formulating our switching methods up to conjugation, we are not losing any generality: a conjugation with a permutation matrix corresponds to an automorphism of the graph.

\subsection{Irreducible switching methods from the infinite family}

Here we consider the irreducible switching methods coming from the \(2m\times2m\)-matrix 
\(R_{2m}=\frac12\:\mathrm{circulant}(J,O,\dots,O,Y)\) from Theorem~\ref{thm:level2mats}(ii) for \(m\in\{3,4,5,6\}\).

From Lemma~\ref{lem:symmetry}, we obtain:

\begin{corollary}
    \(\mathcal{B}_{R_{2m}}\) is closed under conjugation with the permutation matrices \[\mathrm{circulant}(O,I,O,\dots,O)\quad\text{ and }\quad \mathrm{diag}(J-I,I,\dots,I)\] and every matrix that is a product of those.
\end{corollary}
\begin{proof}
    The matrices \(\mathrm{circulant}(O,I,O,\dots,O)\) and \(\mathrm{diag}(J-I,I,\dots,I)\) have the required property of Lemma~\ref{lem:symmetry}. More precisely, the left, block circulant matrix stabilises \(R_{2m}\). The right, block diagonal matrix satisfies \(\mathrm{diag}(J-I,I,\dots,I)^TR_{2m}\mathrm{diag}(I,\dots,I,J-I)=R_{2m}\). Moreover, if two permutation matrices have the property of Lemma~\ref{lem:symmetry}, then so does their product.
\end{proof}

Equation~\eqref{eq:1} gives us a useful tool to generate the elements \(B\in\mathcal{B}_{R_{2m}}\). It suffices to check all \(4\times4\) submatrices 
$$\begin{bmatrix}B_{ij}&B_{i,j+1}\\B_{i+1,j}&B_{i+1,j+1}\end{bmatrix}$$ for which \(B'_{ij}\) is again a 01-matrix, and then the only condition left is that \(B\) consists of a ``union'' of these \(4\times4\) matrices that ``overlap well''. This strategy is also described by Mao and Yan \cite{MaoSimilar}.

Without loss of generality, we may assume that \(B_{11}=0\), which implies that \(B_{ii}=B'_{ii}=O\) for all \(i\). Indeed, if \(B_{11}=0\), then \(B'_{11}=0\) by equation~\eqref{eq:2}, which implies that \(B_{22}\) has constant diagonal sum by equation~\eqref{eq:1} (using the linear independence of the terms). Since \(B_{22}\) has zeroes on the diagonal, \(B_{22}=0\). Induction on \(i\) proves the claim.

The following two observations about $\mathcal{B}_{R_{2m}}$ are also made by Mao and Yan \cite{MaoSimilar}, although the proof of the first one (here Lemma~\ref{lem:even}) is computer assisted, and the second one (here Lemma~\ref{lem:complement}) is posed as a remark without proof.

\begin{lemma}\label{lem:even}
    Let \(B\in\mathcal{B}_{R_{2m}}\). Every block \(B_{ij}\) has an even number of ones.
\end{lemma}
\begin{proof}
    Consider equation \eqref{eq:1}. Then \(a\equiv\mathrm{sum}(B_{ij})\pmod2\), \(b\equiv\mathrm{sum}(B_{i,j+1})\pmod2\) and \(c\equiv\mathrm{sum}(B_{i+1,j})\pmod2\). Taking the sum of the first column of \(4B'_{ij}\) in \eqref{eq:1}, we get that \(2a+2b\equiv0\pmod4\), or equivalently, \(\mathrm{sum}(B_{ij})\equiv\mathrm{sum}(B_{i,j+1})\pmod2\). If we look at the sum of the first row, we get that \(\mathrm{sum}(B_{ij})\equiv\mathrm{sum}(B_{i+1,j})\pmod2\). By induction, all blocks have the same number of ones modulo 2. If \(B\) is an adjacency matrix (in particular, if \(B\in\mathcal{B}_{R_{2m}}\)), then \(B_{11}\) has an even number of ones, and so do all other blocks.
\end{proof}

\begin{lemma}\label{lem:complement}
    \(\mathcal{B}_{R_{2m}}\) is closed under the replacement of any block \(B_{ij}\) (together with its transpose \(B_{ji}\)), \(i\neq j\), by its complement \(J-B_{ij}\) (resp.\ \(J-B_{ji}\)).
    Moreover, the subset of \(\mathcal{B}_{R_{2m}}\) consisting of \emph{irreducible} matrices is also closed under this operation.
\end{lemma}
\begin{proof}
    Consider a block \(B'_{ij}\) of \(B'=R_{2m}^TBR_{2m}\). First, suppose that \(B'_{ij}=O\) or \(B'_{ij}=J\), then \(b=c=d=0\) in equation~\eqref{eq:1}, so \(B'_{ij}=\frac14JB_{ij}J=B_{ij}\). If \(B_{ij}\) gets replaced by its complement, then so does \(B'_{ij}\). If any other block gets replaced, nothing happens to \(B'_{ij}\).
    Now suppose that \(B'_{ij}\in\{I,J-I,N,J-N,N^T,J-N^T\}\). Since the proof for these six cases is similar, we only work out the one where \(B'_{ij}=I\). Then \(a=2\), \(b=c=0\) and \(d=2\) in equation~\eqref{eq:1}. From \(JB_{ij}J=2J\), we get that \(B_{ij}\) has two zeroes. From \(YB_{i+1,j+1}Y=2Y\), we get that \(B_{i+1,j+1}=I\), since it is a 01-matrix. If \(B_{i+1,j+1}\) gets replaced by its complement \(J-I\), then so does \(B'_{ij}\). In all other cases, \(B'_{ij}\) stays the same.
    In general, the replacement of a block of \(B\in\mathcal{B}_{R_{2m}}\) by its complement, results in the replacement of one or more blocks of \(B'\) by their complement.

    To prove the second part of the statement, suppose that \(B\in\mathcal{B}_{R_{2m}}\) is reducible, and let \(C_1\cup\dots\cup C_m\) be the corresponding switching set. By Lemma~\ref{lem:infreducible}, the corresponding switching method is a sequence of switching methods with respect to some \(\mathrm{diag}(R_{2m_1},\dots,R_{2m_d})\) with \(m_1+\dots+m_d\leq m\), whose switching set is a union of \(C_i\)'s. If we replace a block of \(B\) by its complement, then in each step of the sequence, the matrix is mapped to another matrix with one or more blocks replaced by their complement. This does not harm the conditions on the outside vertices. Hence, \(B\) stays reducible under this operation. The claim follows by contraposition. 
\end{proof}

As pointed out in Section~\ref{sec:preliminaries}, the previously known switching methods of level 2 with switching sets of sizes 6, 7 and 8, will be denoted by Six, Seven and Eight vertex AH-switching respectively, following the notation introduced by Abiad and Haemers \cite{AHswitching}.

\subsubsection{Six vertex AH-switching}\label{sec:AH6}

Among the seven possible adjacency matrices for a Six vertex AH-switching set from \cite[Lemma~6]{AHswitching}, only two are irreducible.
However, they are actually equivalent, and correspond to the induced subgraph in the statement of Theorem~\ref{thm:AH6}.
In order to show that the other subgraphs are reducible, we need the following lemma.

\begin{lemma}\label{lem:AH6reducible}
    Let \(B\in\mathcal{B}_{R_6}\). Let the corresponding graph be partitioned as \(C_1\cup C_2\cup C_3\). Then \(B\) is reducible if and only if at least one of the induced subgraphs on $C_1 \cup C_2$, $C_2 \cup C_3$ and $C_1 \cup C_3$ is regular and has the property that the two remaining vertices of \(B\) are adjacent to an even number of its vertices.
\end{lemma}
\begin{proof}
    If \(B\) is reducible, then it is necessarily a product of GM-switching methods on 4 vertices. By Lemma~\ref{lem:infreducible}, these should form one of the sets $C_1 \cup C_2$, $C_2 \cup C_3$ or $C_1 \cup C_3$. This set will meet the conditions of GM-switching if and only if it has the required properties of the statement.
    
    Suppose, w.l.o.g., that $C_1 \cup C_2$ is regular and that the two vertices of $C_3$ have an even number of neighbours in $C_1\cup C_2$. We can write \(R_6\) as the product \[R_6=\begin{bmatrix}R_4&O\\O&I_2\end{bmatrix}\begin{bmatrix}I_2&O\\O&R_4\end{bmatrix}.\]
    We check that GM-switching twice works. Actually, we only have to check that it works the first time (with respect to $C_1\cup C_2$), because the product will automatically result in a valid switching, by definition of \(\mathcal{B}_{R_6}\). The subgraph on $C_1\cup C_2$ is regular by assumption. A vertex in \(C_3\) is adjacent to an even number of vertices in \(C_1\cup C_2\), also by assumption. A vertex outside \(B\) is adjacent to an even number of its vertices. So a conjugation of the adjacency matrix with \(R_6\) is equivalent to GM-switching twice.
\end{proof}

We can now apply the above lemma to rule out possible reducible cases of Six vertex AH-switching. At the same time, we provide a more intuitive, combinatorial description of the method.

\begin{theorem}[Six vertex AH-switching -- irreducible case]\label{thm:AH6}
    Let \(\Gamma\) be a graph with a vertex partition \(\{C_1,C_2,C_3,D\}\) such that:
    \begin{enumerate}[(i)]
        \item \(|C_1|=|C_2|=|C_3|=2\).
        \item Every vertex in \(D\) has the same number of neighbours in \(C_1\), \(C_2\) and \(C_3\) modulo \(2\).
        \item The induced subgraph on \(C_1\cup C_2\cup C_3\) is 
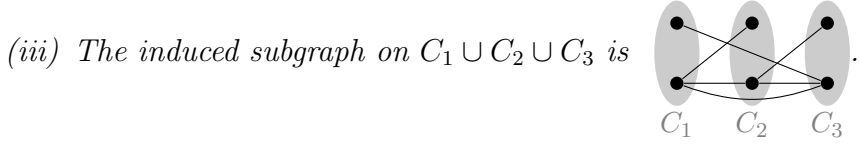
\begin{tikzpicture}[baseline=(O.base)]
    \path (0,-.5) node (O) {};
    \path (-1.5,-.1) rectangle (0,0);
    \fill[ss] (-1,-.4) ellipse (.3 and .7)
    (0,-.4) ellipse (.3 and .7)
    (1,-.4) ellipse (.3 and .7);
    \path[every node/.append style={circle, fill=black, minimum size=5pt, label distance=2pt, inner sep=0pt}]
    (-1,0) node (1) {}
    (-1,-.8) node[label={[label distance=5pt]270:\color{sslabel}\small\(C_1\)}] (2) {}
    (0,0) node (3) {}
    (0,-.8) node[label={[label distance=5pt]270:\color{sslabel}\small\(C_2\)}] (4) {}
    (1,0) node (5) {}
    (1,-.8) node[label={[label distance=5pt]270:\color{sslabel}\small\(C_3\)}] (6) {};
    \draw (2) edge (4) edge (3) edge[out=-20,in=200] (6)
    (4) edge (5) edge (6)
    (6) edge (1);
\end{tikzpicture}.
    \end{enumerate}
    Let \(\pi\) be the permutation on \(C_1\cup C_2\cup C_3\) that shifts the vertices cyclically to the left. For every \(v\in D\) that has exactly one neighbour \(w\) in each \(C_i\), replace each edge \(\{v,w\}\) by \(\{v,\pi(w)\}\).
    Replace the induced subgraph on \(C_1\cup C_2\cup C_3\) by:
    
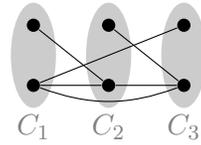
\[\begin{tikzpicture}[baseline=(O.base)]
    \path (0,-.5) node (O) {};
    \path (-1.5,-.1) rectangle (0,0);
    \fill[ss] (-1,-.4) ellipse (.3 and .7)
    (0,-.4) ellipse (.3 and .7)
    (1,-.4) ellipse (.3 and .7);
    \path[every node/.append style={circle, fill=black, minimum size=5pt, label distance=2pt, inner sep=0pt}]
    (-1,0) node (1) {}
    (-1,-.8) node[label={[label distance=5pt]270:\color{sslabel}\small\(C_1\)}] (2) {}
    (0,0) node (3) {}
    (0,-.8) node[label={[label distance=5pt]270:\color{sslabel}\small\(C_2\)}] (4) {}
    (1,0) node (5) {}
    (1,-.8) node[label={[label distance=5pt]270:\color{sslabel}\small\(C_3\)}] (6) {};
    \draw (2) edge (4) edge (5) edge[out=-20,in=200] (6)
    (4) edge (1) edge (6)
    (6) edge (3);
\end{tikzpicture}\]
    The resulting graph is \(\mathbb{R}\)-cospectral with \(\Gamma\).
\end{theorem}
\begin{proof}    
    Six vertex AH-switching is described algebraically by Abiad and Haemers in  \cite[Section~4]{AHswitching}. Our combinatorial description corresponds to the matrix \(B_4\) in \cite[Lemma~6]{AHswitching}. Note that it is equivalent to their matrix \(B_5\) up to conjugation with \(\mathrm{diag}(J-I,I,I)\), cf.\ Lemma~\ref{lem:symmetry}. The other matrices satisfy the condition of Lemma~\ref{lem:AH6reducible} and are therefore reducible.
    In other words, this is the only irreducible switching method that corresponds to a regular orthogonal matrix of level \(2\) with one indecomposable block of size \(6\). To understand the combinatorial interpretation, label the vertices of \(C_i\) as \(C_i=\{v_{2i-1},v_{2i}\}\) and apply Lemma~\ref{lem:phdaida}.
\end{proof}

\begin{figure}[H]
    \centering
\begin{subfigure}{.25\textwidth}
\centering
\begin{tikzpicture}[scale=.8]
    \path (-2.5,-1.5) rectangle (2,1.1);
    \path[every node/.append style={circle, fill=black, minimum size=5pt, label distance=2pt, inner sep=0pt}]
    (-1,0) node (1) {}
    (-1,-1) node (2) {}
    (0,0) node (3) {}
    (0,-1) node (4) {}
    (1,0) node (5) {}
    (1,-1) node (6) {}
    (-2,-.5) node (7) {}
    (0,.9) node (8) {};
    \draw (2) edge (4) edge (3) edge[out=-20,in=200] (6)
    (4) edge (5) edge (6)
    (6) edge (1);
    \draw (7) edge (1) edge (2)
    (8) edge (2) edge (3) edge (5);
\end{tikzpicture}
\end{subfigure}
\begin{subfigure}{.25\textwidth}
\centering
\begin{tikzpicture}[scale=.8]
    \path (-2.5,-1.5) rectangle (2,1.1);
    \path[every node/.append style={circle, fill=black, minimum size=5pt, label distance=2pt, inner sep=0pt}]
    (-1,0) node (1) {}
    (-1,-1) node (2) {}
    (0,0) node (3) {}
    (0,-1) node (4) {}
    (1,0) node (5) {}
    (1,-1) node (6) {}
    (-2,-.5) node (7) {}
    (0,.9) node (8) {};
    \draw (2) edge (4) edge (5) edge[out=-20,in=200] (6)
    (4) edge (1) edge (6)
    (6) edge (3)
    (7) edge (1) edge (2)
    (8) edge (1) edge (3) edge (6);
\end{tikzpicture}
\end{subfigure}
\caption{Cospectral mates, obtained by Six vertex AH-switching.}
\end{figure}
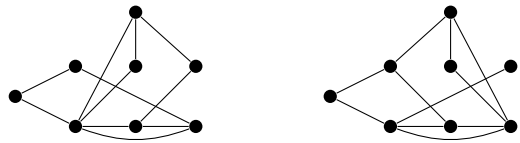

Note that (irreducible) Six vertex AH-switching can be seen as a special case of Theorem~\ref{thm:sun} or Theorem~\ref{thm:infinite1} when \(m=3\).

\subsubsection{Eight vertex AH-switching (from the infinite family)}

In Example~\ref{ex:reducible}, we encountered a reducible Eight vertex AH-switching method. Surprisingly, all 3584 matrices \cite[Section~6]{AHswitching} that describe an Eight vertex AH-switching set corresponding to \(R_8\) (not the sporadic matrix of size 8 in Theorem~\ref{thm:level2mats}(iv)), are reducible: 
\begin{theorem}\label{thm:AH8infreducible}
    All matrices of \(\mathcal{B}_{R_8}\) are reducible. 
\end{theorem}
\begin{proof}
    By computer, see Appendix~\ref{app:code}.
\end{proof}

\subsubsection{Ten vertex switching}

Abiad and Haemers \cite{AHswitching} considered switching operations coming from regular orthogonal matrices with a non-trivial indecomposable block up to size 8. We continue the classification up to size 10 and, in the next section, 12. In line with the names of the previous methods, we call them Ten vertex switching and Twelve vertex switching. The latter is investigated in Section~\ref{sec:12vertices}.


\begin{theorem}\label{thm:AH10}
    \(\mathcal{B}_{R_{10}}\) has exactly three irreducible matrices up to complementation (of both the \(2\times2\) blocks and the full matrix) and conjugation (cf.\ Lemma~\ref{lem:symmetry} and Lemma~\ref{lem:complement}). They correspond to the switching method from Corollary~\ref{cor:10vertex}.
\end{theorem}
\begin{proof}
    By computer, see Appendix~\ref{app:code}.
\end{proof}

Note that these three irreducible switching sets can also be obtained from the infinite families provided by Mao and Yan \cite[Examples~2,3 and 4]{MaoSimilar}.

\subsubsection{Twelve vertex switching}\label{sec:12vertices}
There are 18 different elements of \(\mathcal{B}_{R_{12}}\) up to equivalence (compared to the 3 cases for Ten vertex switching). Since there is no apparent pattern that describes these graphs combinatorially, we present them as matrices in Table~\ref{tab:AH12}. 

\newcommand{\colw}{2mm}
\newcommand{\nt}{N{\rule{0pt}{6pt}}^{\hspace{-.4mm}T}}
\newcommand{\ribrack}{\hspace{1mm}}
\newcommand{\sty}{\arraybackslash$}
\begin{longtable}[H]{>{\arraybackslash$}l<{$}}
\footnotesize 
\left[\def\arraystretch{.9}\begin{array}{>{\sty}p{\colw}<{$}>{\sty}p{\colw}<{$}>{\sty}p{\colw}<{$}>{\sty}p{\colw}<{$}>{\sty}p{\colw}<{$}>{\sty}p{\colw}<{$}}
O&O&I&O&N&\nt\\O&O&O&I&O&N\\I&O&O&N&\nt&O\\O&I&\nt&O&N&O\\\nt&O&N&\nt&O&I\\N&\nt&O&O&I&O
\end{array}\ribrack\right]
\left[\def\arraystretch{.9}\begin{array}{>{\sty}p{\colw}<{$}>{\sty}p{\colw}<{$}>{\sty}p{\colw}<{$}>{\sty}p{\colw}<{$}>{\sty}p{\colw}<{$}>{\sty}p{\colw}<{$}}
O&O&I&I&\nt&I\\O&O&O&I&N&N\\I&O&O&N&N&\nt\\I&I&\nt&O&I&O\\N&\nt&\nt&I&O&I\\I&\nt&N&O&I&O
\end{array}\ribrack\right]
\left[\def\arraystretch{.9}\begin{array}{>{\sty}p{\colw}<{$}>{\sty}p{\colw}<{$}>{\sty}p{\colw}<{$}>{\sty}p{\colw}<{$}>{\sty}p{\colw}<{$}>{\sty}p{\colw}<{$}}
O&O&I&I&N&\nt\\O&O&O&I&I&N\\I&O&O&N&\nt&I\\I&I&\nt&O&N&O\\\nt&I&N&\nt&O&I\\N&\nt&I&O&I&O
\end{array}\ribrack\right]
\left[\def\arraystretch{.9}\begin{array}{>{\sty}p{\colw}<{$}>{\sty}p{\colw}<{$}>{\sty}p{\colw}<{$}>{\sty}p{\colw}<{$}>{\sty}p{\colw}<{$}>{\sty}p{\colw}<{$}}
O&O&I&\nt&N&O\\O&O&O&N&N&\nt\\I&O&O&N&O&O\\N&\nt&\nt&O&I&O\\\nt&\nt&O&I&O&I\\O&N&O&O&I&O
\end{array}\ribrack\right]
\vspace{3mm}
\\
\footnotesize 
\left[\def\arraystretch{.9}\begin{array}{>{\sty}p{\colw}<{$}>{\sty}p{\colw}<{$}>{\sty}p{\colw}<{$}>{\sty}p{\colw}<{$}>{\sty}p{\colw}<{$}>{\sty}p{\colw}<{$}}
O&O&I&\nt&N&O\\O&O&O&N&N&\nt\\I&O&O&N&\nt&O\\N&\nt&\nt&O&N&\nt\\\nt&\nt&N&\nt&O&N\\O&N&O&N&\nt&O
\end{array}\ribrack\right]
\left[\def\arraystretch{.9}\begin{array}{>{\sty}p{\colw}<{$}>{\sty}p{\colw}<{$}>{\sty}p{\colw}<{$}>{\sty}p{\colw}<{$}>{\sty}p{\colw}<{$}>{\sty}p{\colw}<{$}}
O&O&N&\nt&I&\nt\\O&O&O&N&\nt&N\\\nt&O&O&O&N&N\\N&\nt&O&O&N&\nt\\I&N&\nt&\nt&O&O\\N&\nt&\nt&N&O&O
\end{array}\ribrack\right]
\left[\def\arraystretch{.9}\begin{array}{>{\sty}p{\colw}<{$}>{\sty}p{\colw}<{$}>{\sty}p{\colw}<{$}>{\sty}p{\colw}<{$}>{\sty}p{\colw}<{$}>{\sty}p{\colw}<{$}}
O&O&N&\nt&N&\nt\\O&O&O&N&\nt&N\\\nt&O&O&O&N&N\\N&\nt&O&O&N&\nt\\\nt&N&\nt&\nt&O&N\\N&\nt&\nt&N&\nt&O
\end{array}\ribrack\right]
\left[\def\arraystretch{.9}\begin{array}{>{\sty}p{\colw}<{$}>{\sty}p{\colw}<{$}>{\sty}p{\colw}<{$}>{\sty}p{\colw}<{$}>{\sty}p{\colw}<{$}>{\sty}p{\colw}<{$}}
O&O&I&I&I&\nt\\O&O&N&\nt&I&N\\I&\nt&O&O&I&I\\I&N&O&O&N&\nt\\I&I&I&\nt&O&O\\N&\nt&I&N&O&O
\end{array}\ribrack\right]
\vspace{3mm}
\\
\footnotesize 
\left[\def\arraystretch{.9}\begin{array}{>{\sty}p{\colw}<{$}>{\sty}p{\colw}<{$}>{\sty}p{\colw}<{$}>{\sty}p{\colw}<{$}>{\sty}p{\colw}<{$}>{\sty}p{\colw}<{$}}
O&O&I&N&\nt&\nt\\O&O&N&\nt&\nt&N\\I&\nt&O&O&N&O\\\nt&N&O&O&N&\nt\\N&N&\nt&\nt&O&O\\N&\nt&O&N&O&O
\end{array}\ribrack\right]
\left[\def\arraystretch{.9}\begin{array}{>{\sty}p{\colw}<{$}>{\sty}p{\colw}<{$}>{\sty}p{\colw}<{$}>{\sty}p{\colw}<{$}>{\sty}p{\colw}<{$}>{\sty}p{\colw}<{$}}
O&O&I&O&N&\nt\\O&O&N&\nt&O&N\\I&\nt&O&N&O&O\\O&N&\nt&O&I&O\\\nt&O&O&I&O&I\\N&\nt&O&O&I&O
\end{array}\ribrack\right]
\left[\def\arraystretch{.9}\begin{array}{>{\sty}p{\colw}<{$}>{\sty}p{\colw}<{$}>{\sty}p{\colw}<{$}>{\sty}p{\colw}<{$}>{\sty}p{\colw}<{$}>{\sty}p{\colw}<{$}}
O&O&I&I&N&\nt\\O&O&N&\nt&I&N\\I&\nt&O&N&O&I\\I&N&\nt&O&I&O\\\nt&I&O&I&O&I\\N&\nt&I&O&I&O
\end{array}\ribrack\right]
\left[\def\arraystretch{.9}\begin{array}{>{\sty}p{\colw}<{$}>{\sty}p{\colw}<{$}>{\sty}p{\colw}<{$}>{\sty}p{\colw}<{$}>{\sty}p{\colw}<{$}>{\sty}p{\colw}<{$}}
O&O&N&N&\nt&\nt\\O&O&N&\nt&\nt&N\\\nt&\nt&O&O&N&N\\\nt&N&O&O&N&\nt\\N&N&\nt&\nt&O&O\\N&\nt&\nt&N&O&O
\end{array}\ribrack\right]
\vspace{3mm}
\\
\footnotesize 
\left[\def\arraystretch{.9}\begin{array}{>{\sty}p{\colw}<{$}>{\sty}p{\colw}<{$}>{\sty}p{\colw}<{$}>{\sty}p{\colw}<{$}>{\sty}p{\colw}<{$}>{\sty}p{\colw}<{$}}
O&I&O&I&N&\nt\\I&O&I&N&\nt&\nt\\O&I&O&I&\nt&N\\I&\nt&I&O&N&\nt\\\nt&N&N&\nt&O&N\\N&N&\nt&N&\nt&O
\end{array}\ribrack\right]
\left[\def\arraystretch{.9}\begin{array}{>{\sty}p{\colw}<{$}>{\sty}p{\colw}<{$}>{\sty}p{\colw}<{$}>{\sty}p{\colw}<{$}>{\sty}p{\colw}<{$}>{\sty}p{\colw}<{$}}
O&I&N&\nt&N&\nt\\I&O&I&N&\nt&\nt\\\nt&I&O&I&\nt&N\\N&\nt&I&O&N&O\\\nt&N&N&\nt&O&I\\N&N&\nt&O&I&O
\end{array}\ribrack\right]
\left[\def\arraystretch{.9}\begin{array}{>{\sty}p{\colw}<{$}>{\sty}p{\colw}<{$}>{\sty}p{\colw}<{$}>{\sty}p{\colw}<{$}>{\sty}p{\colw}<{$}>{\sty}p{\colw}<{$}}
O&I&N&\nt&N&\nt\\I&O&I&N&\nt&\nt\\\nt&I&O&I&\nt&N\\N&\nt&I&O&N&\nt\\\nt&N&N&\nt&O&N\\N&N&\nt&N&\nt&O
\end{array}\ribrack\right]
\left[\def\arraystretch{.9}\begin{array}{>{\sty}p{\colw}<{$}>{\sty}p{\colw}<{$}>{\sty}p{\colw}<{$}>{\sty}p{\colw}<{$}>{\sty}p{\colw}<{$}>{\sty}p{\colw}<{$}}
O&I&\nt&O&N&\nt\\I&O&N&O&O&O\\N&\nt&O&I&\nt&O\\O&O&I&O&N&O\\\nt&O&N&\nt&O&I\\N&O&O&O&I&O
\end{array}\ribrack\right]
\vspace{3mm}
\\
\footnotesize 
\left[\def\arraystretch{.9}\begin{array}{>{\sty}p{\colw}<{$}>{\sty}p{\colw}<{$}>{\sty}p{\colw}<{$}>{\sty}p{\colw}<{$}>{\sty}p{\colw}<{$}>{\sty}p{\colw}<{$}}
O&I&\nt&N&N&\nt\\I&O&N&O&I&\nt\\N&\nt&O&I&\nt&N\\\nt&O&I&O&N&N\\\nt&I&N&\nt&O&I\\N&N&\nt&\nt&I&O
\end{array}\ribrack\right]
\left[\def\arraystretch{.9}\begin{array}{>{\sty}p{\colw}<{$}>{\sty}p{\colw}<{$}>{\sty}p{\colw}<{$}>{\sty}p{\colw}<{$}>{\sty}p{\colw}<{$}>{\sty}p{\colw}<{$}}
O&I&\nt&N&N&\nt\\I&O&N&N&\nt&\nt\\N&\nt&O&I&\nt&N\\\nt&\nt&I&O&N&N\\\nt&N&N&\nt&O&I\\N&N&\nt&\nt&I&O
\end{array}\ribrack\right]
\\
\caption{Irreducible adjacency matrices \(B_1,\dots,B_{18}\in\mathcal{B}_{R_{12}}\).}
    \label{tab:AH12}
\end{longtable}

Note that none of the matrices in Table~\ref{tab:AH12} are block circulant. In particular, the new switching method from Theorem~\ref{thm:infinite2} is reducible if \(m=6\).
The matrices \(B'=R_{12}^TBR_{12}\) can be obtained using equation~\eqref{eq:1}.

Next, we provide a more combinatorial statement of Twelve vertex switching.

\begin{theorem}[Twelve vertex switching -- irreducible cases]\label{thm:12vertex}
    Let \(\Gamma\) be a graph with a vertex partition \(\{C_1,\dots,C_6,D\}\) such that:
    \begin{enumerate}[(i)]
        \item \(|C_1|=\dots=|C_6|=2\).
        \item Every vertex in \(D\) has the same number of neighbours in \(C_1,\dots,C_6\) modulo \(2\).
        \item The adjacency matrix \(B\) of the induced subgraph on \(C_1\cup\dots\cup C_6\) is one of the 18 matrices of Table~\ref{tab:AH12}, where any \(2\times2\) block, together with its transpose, may be replaced by its complement.
    \end{enumerate}
    Let \(\pi\) be the permutation on \(C_1\cup\dots\cup C_6\) that shifts the vertices cyclically to the left. For every \(v\notin C\) that has exactly one neighbour \(w\) in each \(C_i\), replace each edge \(\{v,w\}\) by \(\{v,\pi(w)\}\). Replace the induced subgraph on \(C_1\cup\dots\cup C_6\) by the one with adjacency matrix \(R_{12}^TBR_{12}\). The resulting graph is \(\mathbb{R}\)-cospectral with \(\Gamma\).
\end{theorem}
\begin{proof}
    By computer, see Appendix~\ref{app:code}. The code from the appendix shows that these are all irreducible switchings with one indecomposable block of size 12. The conditions on the adjacencies with the vertices outside \(C\) are given by Lemma~\ref{lem:phdaida}.
\end{proof}

\subsection{Fano switching (Seven vertex AH-switching)}\label{sec:fano}

In this section, we consider the switching method corresponding to the matrix \(R_\text{Fano}=\frac12\:\mathrm{circulant}(-1,1,1,0,1,0,0)\) from Theorem~\ref{thm:level2mats}(iii).
We call it Fano switching, because of its geometric interpretation, that we provide in Theorem~\ref{thm:fano}.

\begin{lemma}[{\cite[Lemma~7]{AHswitching}}]\label{lem:fano}
    Let \(P\) be the cyclic permutation matrix \(P=\mathrm{cycle}(0100000)\). Then \(v\in\mathcal{V}_R\) if and only if \(v\) is equal to \(0\), \(\mathbb{1}\), \(P^i(1101000)^T\) or \(P^i(0010111)^T\) for some \(i\in\{0,\dots,6\}\). If \(v=P^i(1101000)^T\), then \(R_\text{Fano}^Tv=P^i(0010110)^T\). If \(v=P^i(0010111)^T\), then \(R_\text{Fano}^Tv=P^i(1101001)^T\).
\end{lemma}

We describe only the irreducible cases, up to complementation and conjugation.

\begin{theorem}[Fano switching -- irreducible cases]\label{thm:fano}
    Let \(\Gamma\) be a graph and \(C=\{v_1,\dots,v_7\}\) a subset of its vertices. Let \(\pi\) be the cyclic permutation \((v_i\mapsto v_{i+1\pmod{7}})\). Define \(\ell:=\{v_1,v_2,v_4\}\), \(\mathcal{O}:=\{v_3,v_5,v_6\}\), and \(\ell_i=\pi^i(\ell)\) and \(\mathcal{O}_i=\pi^i(\mathcal{O})\) for all \(i\in\mathbb{Z}/7\mathbb{Z}\). The set \(C\), together with the ``lines'' \(\ell_i\), form the Fano plane \(\mathrm{PG}(2,2)\).
    Suppose that: 
    \begin{enumerate}[(i)]
        \item The induced subgraph on \(C\) is (a) \(\newcommand{\radius}{1.2}
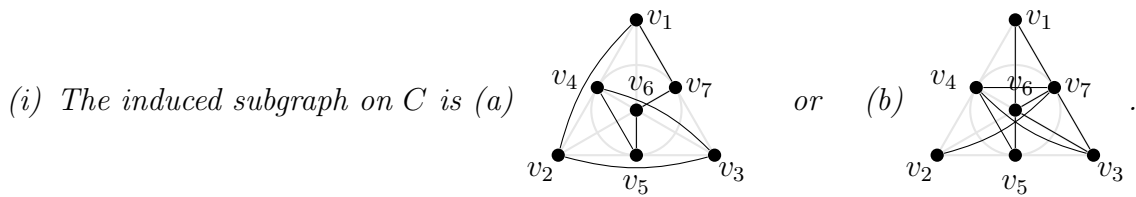
\begin{tikzpicture}[baseline=(O.base)]
    \path (0,0) node (O) {};
    \draw[thick,backline] (-30:\radius) -- coordinate (P1)
    (90:\radius) -- coordinate (P2)
    (210:\radius) -- coordinate (P3) cycle;
    \draw[thick,backline] (210:\radius) -- (P1) (-30:\radius) -- (P2) (90:\radius) -- (P3);
    \node[draw,thick,backline] at (O) [circle through=(P1)] {};
    \path[every node/.append style={circle, fill=black, minimum size=5pt, label distance=0pt, inner sep=0pt}]
    (O) node[label={[xshift=2pt]90:\(v_6\)}] (6) {}
    (P1) node[label={0:\(v_7\)}] (7) {}
    (P2) node[label={[label distance=3pt,yshift=3pt]180:\(v_4\)}] (4) {}
    (P3) node[label={[label distance=2pt]270:\(v_5\)}] (5) {}
    (-30:\radius) node[label={-40:\(v_3\)}] (3) {}
    (90:\radius) node[label={[label distance=0pt]0:\(v_1\)}] (1) {}
    (210:\radius) node[label={[label distance=0pt]220:\(v_2\)}] (2) {};
    \draw (1) edge[out=-135,in=75] (2)
    (2) edge[out=-15,in=195] (3)
    (3) edge[out=135,in=-15] (4)
    (4) edge (5)
    (5) edge (6)
    (6) edge (7)
    (7) edge (1);
\end{tikzpicture}\quad\text{ or }\quad\text{(b)}
\begin{tikzpicture}[baseline=(O.base)]
    \path (0,0) node (O) {};
    \draw[thick,backline] (-30:\radius) -- coordinate (P1)
    (90:\radius) -- coordinate (P2)
    (210:\radius) -- coordinate (P3) cycle;
    \draw[thick,backline] (210:\radius) -- (P1) (-30:\radius) -- (P2) (90:\radius) -- (P3);
    \node[draw,thick,backline] at (O) [circle through=(P1)] {};
    \path[every node/.append style={circle, fill=black, minimum size=5pt, label distance=0pt, inner sep=0pt}]
    (O) node[label={[xshift=2pt]90:\(v_6\)}] (6) {}
    (P1) node[label={0:\(v_7\)}] (7) {}
    (P2) node[label={[label distance=3pt,yshift=3pt]180:\(v_4\)}] (4) {}
    (P3) node[label={[label distance=2pt]270:\(v_5\)}] (5) {}
    (-30:\radius) node[label={-40:\(v_3\)}] (3) {}
    (90:\radius) node[label={[label distance=0pt]0:\(v_1\)}] (1) {}
    (210:\radius) node[label={[label distance=0pt]220:\(v_2\)}] (2) {};
    \draw (4) edge[out=-45,in=165] (3) edge (5)
    (6) edge (1) edge (3) edge (5)
    (7) edge (1) edge[out=225,in=15] (2) edge (3) edge (4) edge (6);
\end{tikzpicture}\).
        \item Every vertex outside \(C\) is adjacent to either:
        \begin{enumerate}[1.]
            \item All vertices of \(C\).
            \item No vertices of \(C\).
            \item Three vertices of \(C\) contained in a line.
            \item Four vertices of \(C\) not contained in a line.
        \end{enumerate}
    \end{enumerate}
    For every \(v\) outside \(C\) that is (non)adjacent to the three vertices of a line \(\ell_i\), make it (non)adjacent to the three vertices of \(\mathcal{O}_i\). In case (b) above, replace the induced subgraph on \(C\) by: \[\newcommand{\radius}{1.2}
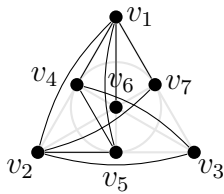
\begin{tikzpicture}[baseline=(O.base)]
    \path (0,0) node (O) {};
    \draw[thick,backline] (-30:\radius) -- coordinate (P1)
    (90:\radius) -- coordinate (P2)
    (210:\radius) -- coordinate (P3) cycle;
    \draw[thick,backline] (210:\radius) -- (P1) (-30:\radius) -- (P2) (90:\radius) -- (P3);
    \node[draw,thick,backline] at (O) [circle through=(P1)] {};
    \path[every node/.append style={circle, fill=black, minimum size=5pt, label distance=0pt, inner sep=0pt}]
    (O) node[label={[xshift=2pt]90:\(v_6\)}] (6) {}
    (P1) node[label={0:\(v_7\)}] (7) {}
    (P2) node[label={[label distance=3pt,yshift=3pt]180:\(v_4\)}] (4) {}
    (P3) node[label={[label distance=2pt]270:\(v_5\)}] (5) {}
    (-30:\radius) node[label={-40:\(v_3\)}] (3) {}
    (90:\radius) node[label={[label distance=0pt]0:\(v_1\)}] (1) {}
    (210:\radius) node[label={[label distance=0pt]220:\(v_2\)}] (2) {};
    \draw (4) edge[out=-15,in=135] (3) edge (5)
    (2) edge[out=-15,in=195] (3) edge (5) edge[out=15,in=-135] (7)
    (1) edge[out=-135,in=75] (2) edge (4) edge[out=-105,in=105] (5) edge (6) edge (7);
\end{tikzpicture}\] The resulting graph is \(\mathbb{R}\)-cospectral with \(\Gamma\).
\end{theorem}
\begin{proof}
    In \cite[Section~5]{AHswitching}, Abiad and Haemers obtained 12 options for \(B\in\mathcal{B}_{R_\text{Fano}}\). By computer (see Appendix~\ref{app:code}), one verifies that only 2 of them are irreducible. The induced subgraphs correspond to the matrices \(B_1\) (up to permutation) and \(B_3\) in the proof of \cite[Lemma~8]{AHswitching}. The conditions on the vertices outside \(C\) are given by Lemma~\ref{lem:fano}.
\end{proof}

Note that the sets \(\mathcal{O}_i\) are so-called ovals of the Fano plane. They form again a Fano plane, but the lines are permuted in a non-trivial way.

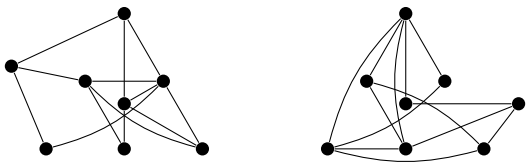
\begin{figure}[H]
    \centering
\begin{subfigure}{.2\textwidth}
\centering
\newcommand{\radius}{1.2}
\begin{tikzpicture}[baseline=(O.base)]
    \path (0,0) node (O) {};
    \path[thick,backline] (-30:\radius) -- coordinate (P1)
    (90:\radius) -- coordinate (P2)
    (210:\radius) -- coordinate (P3) cycle;
    \path[thick,backline] (210:\radius) -- (P1) (-30:\radius) -- (P2) (90:\radius) -- (P3);
    \node[thick,backline] at (O) [circle through=(P1)] {};
    \path[every node/.append style={circle, fill=black, minimum size=5pt, label distance=0pt, inner sep=0pt}]
    (O) node (6) {}
    (P1) node (7) {}
    (P2) node (4) {}
    (P3) node (5) {}
    (-30:\radius) node (3) {}
    (90:\radius) node (1) {}
    (210:\radius) node (2) {}
    (-1.5,.5) node (L) {};
    \draw (4) edge[out=-45,in=165] (3) edge (5)
    (6) edge (1) edge (3) edge (5)
    (7) edge (1) edge[out=225,in=15] (2) edge (3) edge (4) edge (6)
    (L) edge (1) edge (2) edge (4);
\end{tikzpicture}
\end{subfigure}
\qquad
\begin{subfigure}{.2\textwidth}
\centering
\newcommand{\radius}{1.2}
\begin{tikzpicture}[baseline=(O.base)]
    \path (0,0) node (O) {};
    \path[thick,backline] (-30:\radius) -- coordinate (P1)
    (90:\radius) -- coordinate (P2)
    (210:\radius) -- coordinate (P3) cycle;
    \path[thick,backline] (210:\radius) -- (P1) (-30:\radius) -- (P2) (90:\radius) -- (P3);
    \node[thick,backline] at (O) [circle through=(P1)] {};
    \path[every node/.append style={circle, fill=black, minimum size=5pt, label distance=0pt, inner sep=0pt}]
    (O) node (6) {}
    (P1) node (7) {}
    (P2) node (4) {}
    (P3) node (5) {}
    (-30:\radius) node (3) {}
    (90:\radius) node (1) {}
    (210:\radius) node (2) {}
    (1.5,0) node (R) {};
    \draw (4) edge[out=-15,in=135] (3) edge (5)
    (2) edge[out=-15,in=195] (3) edge (5) edge[out=15,in=-135] (7)
    (1) edge[out=-135,in=75] (2) edge (4) edge[out=-105,in=105] (5) edge (6) edge (7)
    (R) edge (3) edge (5) edge (6);
\end{tikzpicture}
\end{subfigure}
\caption{Cospectral mates, obtained by Fano switching.}
\label{fig:fano}
\end{figure}

The example in Figure~\ref{fig:fano} is the combinatorial equivalent of the (algebraic) example given by Abiad and Haemers at the end of \cite[Section~5]{AHswitching}.

The description of the reducible cases is the same, except that the subgraph \(C\) should be a graph whose adjacency matrix is one of the 10 other matrices mentioned in \cite[Lemma~8]{AHswitching}. For example, it can be the empty graph (\(B=O\)).

\begin{example}
    We illustrate that Fano switching with respect to an empty graph (a coclique) is reducible. Label the vertices as in Theorem~\ref{thm:fano}. One checks that the same operation as Fano switching can be obtained by first applying Six vertex AH-switching on the pairs \(\{v_2,v_4\}\), \(\{v_5,v_6\}\) and \(\{v_3,v_7\}\) (in that order) and then GM-switching with respect to \(\{v_1,v_4,v_6,v_7\}\). After permuting the vertex labels with \((265734)\), we get the same result.
    \newcommand{\radius}{1.5}
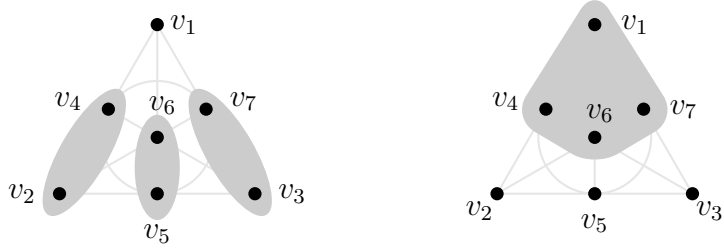
\[
\begin{tikzpicture}[baseline=(O.base)]
    \path (0,0) node (O) {};
    \draw[thick,backline] (-30:\radius) -- coordinate (P1)
    (90:\radius) -- coordinate (P2)
    (210:\radius) -- coordinate (P3) cycle;
    \draw[thick,backline] (210:\radius) -- (P1) (-30:\radius) -- (P2) (90:\radius) -- (P3);
    \node[draw,thick,backline] at (O) [circle through=(P1)] {};
    \fill[ss,rotate around={-30:(-.97,-.2)}] (-.97,-.2) ellipse (.32 and .96);
    \fill[ss] (0,-.4) ellipse (.3 and .7);
    \fill[ss,rotate around={30:(.97,-.2)}] (.97,-.2) ellipse (.32 and .96);
    \path[every node/.append style={circle, fill=black, minimum size=5pt, label distance=0pt, inner sep=0pt}]
    (O) node[label={[xshift=2pt,label distance=3pt]90:\(v_6\)}] (6) {}
    (P1) node[label={[label distance=5pt,yshift=3pt]0:\(v_7\)}] (7) {}
    (P2) node[label={[label distance=6pt,yshift=3pt]180:\(v_4\)}] (4) {}
    (P3) node[label={[label distance=6pt]270:\(v_5\)}] (5) {}
    (-30:\radius) node[label={[label distance=5pt]0:\(v_3\)}] (3) {}
    (90:\radius) node[label={[label distance=1pt]0:\(v_1\)}] (1) {}
    (210:\radius) node[label={[label distance=5pt]180:\(v_2\)}] (2) {};
\end{tikzpicture}
\qquad\qquad
\begin{tikzpicture}[baseline=(O.base)]
    \path (0,0) node (O) {};
    \draw[thick,backline] (-30:\radius) -- coordinate (P1)
    (90:\radius) -- coordinate (P2)
    (210:\radius) -- coordinate (P3) cycle;
    \draw[thick,backline] (210:\radius) -- (P1) (-30:\radius) -- (P2) (90:\radius) -- (P3);
    \node[draw,thick,backline] at (O) [circle through=(P1)] {};
    \fill[ss,rounded corners=10pt] (90:\radius+.5) -- (167:\radius-.37) -- (0,-.4) -- (13:\radius-.37) -- cycle;
    \path[every node/.append style={circle, fill=black, minimum size=5pt, label distance=0pt, inner sep=0pt}]
    (O) node[label={[xshift=2pt]90:\(v_6\)}] (6) {}
    (P1) node[label={[label distance=6pt]0:\(v_7\)}] (7) {}
    (P2) node[label={[label distance=6pt,yshift=3pt]180:\(v_4\)}] (4) {}
    (P3) node[label={[label distance=2pt]270:\(v_5\)}] (5) {}
    (-30:\radius) node[label={-40:\(v_3\)}] (3) {}
    (90:\radius) node[label={[label distance=6pt]0:\(v_1\)}] (1) {}
    (210:\radius) node[label={[label distance=0pt]220:\(v_2\)}] (2) {};
\end{tikzpicture}
\]
    For example, if a vertex is adjacent with \(v_1\), \(v_2\) and \(v_4\), then Six vertex AH-switching leaves the adjacencies invariant. After GM-switching, its new neighbours are \(v_2\), \(v_6\) and \(v_7\), that is, \(v_3\), \(v_5\) and \(v_6\) after relabelling. Similarly, each vertex that is adjacent to the line \(\ell_i\) becomes adjacent to \(\mathcal{O}_i\).
\end{example}

In general, GM-switching on four vertices, not on a line, and Six vertex AH-switching on three pairs of vertices coming from the three lines through a point, without that given point, are the only two smaller methods that can occur if Fano switching is reducible. Moreover, they do not change the conditions on the adjacencies with vertices outside the switching set, since lines are mapped to lines.

This new interpretation of Seven vertex AH-switching leads to a first application of the method. It provides an alternative proof that the 2-Kneser graph \(K_2(n,k)\) is not determined by its spectrum \cite[Theorem~15]{johnson}. This graph has as vertices the \(k\)-subspaces of \(\mathbb{F}_2^n\), where two vertices are adjacent if they intersect trivially.
Using Fano switching, we can provide an alternative proof argument, hence shorting the first part of the proof given by Abiad, D'haeseleer, Haemers and Simoens \cite[Theorem~15]{johnson} that was originally based on GM-switching.
\begin{theorem}[{\cite[Theorem~15]{johnson}}]\label{thm:2kneser}
    \(K_2(n,k)\) is not determined by its spectrum.
\end{theorem}
\begin{proof}
    Let \(\pi\) be a \(3\)-space (Fano plane) and let \(\sigma\) be a \((k-2)\)-subspace of \(\mathbb{F}_2^n\) (\(PG(n-1,2)\)) that intersects \(\pi\) trivially. Define \(C\) as the set of \(k\)-spaces spanned by \(\sigma\) and a line of \(\pi\). Then \(C\) is a Fano switching set since the Fano plane is self-dual. 

    Fano switching with respect to this set will send some collineair points to non-collinear points. If we take three such points and call them \(p_1,p_2\) and \(p_4\), then one can continue the proof of \cite[Theorem~15]{johnson} from paragraph 3 ``Let \(\Gamma'\) be the graph obtained...'' onwards.
\end{proof}
The switching set \(C\) in the above proof is a coclique, and hence reducible by Theorem~\ref{thm:fano}. In fact, it can be obtained by a sequence of repeated GM-switching, which is the method used in the original proof from \cite[Theorem~15]{johnson}. However, the connection to the Fano plane is more directly obvious. This illustrates that reducible switching methods can still be useful.

\begin{remark}
    It is possible to extend Fano switching to a class of switching methods based on projective planes. 
    If the plane has order \(q\), then the corresponding matrix has level \(q\) and size \(q^2+q+1\). The existence of such an orthogonal matrix was first proved by Wallis and Whiteman \cite{weighingprimeorder}. An alternative, elegant proof is due to Kovács and can be found in the thesis of Hain \cite[Theorem~16]{thesishain}. This has been recently further explored in \cite{designswitching}. In this paper however, we stick to switching methods of level 2.
\end{remark}

\subsection{Cube switching (sporadic Eight vertex AH-switching)}\label{sec:cube}

Here we consider the matrix
\[R_\text{cube}=\frac{1}{2}\left[\arraycolsep=3pt\def\arraystretch{.9}\begin{array}{rrrr}
-I & I & I & I \\
I & -Z & I & Z \\
I & Z & -Z & I \\
I & I & Z & -Z
\end{array}\right]\]
from Theorem~\ref{thm:level2mats}(iv) (and not the matrix of order 8 that is part of the infinite family in Theorem~\ref{thm:level2mats}(ii)). Because of our combinatorial reinterpretation in Lemma~\ref{lem:cube} and Theorem~\ref{thm:cube} below, we call it Cube switching. Surprisingly, just like the Eight vertex AH-switching from the infinite family in Theorem~\ref{thm:level2mats}(ii), it is always reducible.

\begin{theorem}\label{thm:AH8sporreducible}
    All matrices of \(\mathcal{B}_{R_\text{cube}}\) are reducible.
\end{theorem}
\begin{proof}
    By computer (see Appendix~\ref{app:code}).
\end{proof}

Still, it can be useful to describe the method geometrically, since it might be more intuitive to use in applications, just like for Theorem~\ref{thm:2kneser} with Fano switching. Moreover, in this way, we complete the description of Eight vertex AH-switching that was started by Abiad and Haemers in \cite[Section~6]{AHswitching}.

\begin{lemma}\label{lem:cube}
    Label the vertices of a cube as follows:
    \vspace{-3mm}
    \begin{figure}[H]
        \centering
        \newcommand{\radius}{.8}
\newcommand{\shift}{.4}
\begin{tikzpicture}[baseline=(O.base)]
    \path (1.5*\radius-.5*\shift,1.5*\radius) node (O) {};
    \path (-.5*\radius-.5*\shift,-.5*\radius) rectangle (3.5*\radius-.5*\shift,3.5*\radius);
    \path[every node/.append style={circle, fill=white, minimum size=15pt, label distance=-15pt, inner sep=0pt}]
    (2*\radius,0) node[label=3] (101) {}
    (2*\radius,2*\radius) node[label=6] (001) {}
    (3*\radius-\shift,\radius) node[label=2] (111) {}
    (3*\radius-\shift,3*\radius) node[label=7] (011) {}
    (0,0) node[label=8] (100) {}
    (0,2*\radius) node[label=1] (000) {}
    (\radius-\shift,\radius) node[label=5] (110) {}
    (\radius-\shift,3*\radius) node[label=4] (010) {};
    \draw (000) edge (100) edge (010) edge (001)
    (011) edge (111) edge (001) edge (010)
    (101) edge (001) edge (111) edge (100)
    (110) edge (010) edge (100) edge (111);
\end{tikzpicture}
    \end{figure}
    \vspace{-4mm}
    \noindent We have that \(v\in\mathcal{V}_R\) if and only if \(v=0\) or \(v=\mathbb{1}\) or \(v\) is the characteristic vector of an affine plane of the cube, seen as \(\mathrm{AG}(3,2)\). That is, it is either a face of the cube, a plane connecting two opposite edges, or a set of four non-adjacent vertices.
    Let \(\pi\) be the permutation \((12)(385476)\). If \(v\) is the characteristic vector of a face, then \(R_\text{cube}^Tv=\pi(v)\). If \(v\) corresponds to a plane connecting opposite edges, then \(R_\text{cube}^Tv=\mathbb{1}-v\). If it describes a set of four non-adjacent vertices, then \(R_\text{cube}^Tv=v\).
\end{lemma}
\begin{proof}
    This follows by straightforward verification. There are 16 cases to consider. See Appendix~\ref{app:code} for a verification by computer.
\end{proof}

This description intuitively tells us that some of the cases of cube switching are reducible to a product involving Fano switching, since fixing one point on the cube and looking at the planes through it, is the same as looking at the lines of the Fano plane. Moreover, the admissible vectors \(v\) are the codewords of the extended Hamming code, a self-dual linear \([8,4,4]\)-code, which can be derived from the incidence matrix of the Fano plane.

Abiad and Haemers \cite[Section~6]{AHswitching} already noted that \(\mathcal{B}_{R_\text{cube}}\) contains 1504 elements, which is too much to describe. However, up to complementation and conjugation (see Lemma~\ref{lem:symmetry}), there are only 40, and they are all reducible. 

\newcommand{\colsep}{1.4pt}
\newcommand{\stret}{.5}
\begin{longtable}[H]{>{\arraybackslash$}c<{$}}
\small 
\left[\arraycolsep=\colsep\def\arraystretch{\stret}\begin{array}{cccccccc}
0&0&0&0&0&0&0&0\\0&0&0&0&1&1&1&0\\0&0&0&0&0&1&0&1\\0&0&0&0&1&0&1&1\\0&1&0&1&0&0&0&0\\0&1&1&0&0&0&1&0\\0&1&0&1&0&1&0&1\\0&0&1&1&0&0&1&0
\end{array}\right]
\left[\arraycolsep=\colsep\def\arraystretch{\stret}\begin{array}{cccccccc}
0&0&0&0&0&0&0&0\\0&0&0&1&1&0&1&0\\0&0&0&1&0&1&0&1\\0&1&1&0&1&0&0&1\\0&1&0&1&0&1&0&0\\0&0&1&0&1&0&0&0\\0&1&0&0&0&0&0&1\\0&0&1&1&0&0&1&0
\end{array}\right]
\left[\arraycolsep=\colsep\def\arraystretch{\stret}\begin{array}{cccccccc}
0&0&0&0&0&0&0&0\\0&0&0&1&1&1&1&0\\0&0&0&1&0&1&0&1\\0&1&1&0&1&1&0&1\\0&1&0&1&0&0&0&0\\0&1&1&1&0&0&1&0\\0&1&0&0&0&1&0&1\\0&0&1&1&0&0&1&0
\end{array}\right]
\left[\arraycolsep=\colsep\def\arraystretch{\stret}\begin{array}{cccccccc}
0&0&0&0&0&0&0&0\\0&0&1&0&1&0&1&0\\0&1&0&1&0&0&0&0\\0&0&1&0&0&0&0&0\\0&1&0&0&0&1&0&0\\0&0&0&0&1&0&0&0\\0&1&0&0&0&0&0&1\\0&0&0&0&0&0&1&0
\end{array}\right]
\left[\arraycolsep=\colsep\def\arraystretch{\stret}\begin{array}{cccccccc}
0&0&0&0&0&0&0&0\\0&0&1&0&1&0&1&1\\0&1&0&1&0&0&0&1\\0&0&1&0&0&0&0&0\\0&1&0&0&0&1&0&0\\0&0&0&0&1&0&0&1\\0&1&0&0&0&0&0&0\\0&1&1&0&0&1&0&0
\end{array}\right]
\vspace{3mm}
\\
\small 
\left[\arraycolsep=\colsep\def\arraystretch{\stret}\begin{array}{cccccccc}
0&0&0&0&0&0&0&0\\0&0&1&0&1&1&1&1\\0&1&0&1&0&0&0&1\\0&0&1&0&0&1&0&0\\0&1&0&0&0&0&0&0\\0&1&0&1&0&0&1&1\\0&1&0&0&0&1&0&0\\0&1&1&0&0&1&0&0
\end{array}\right]
\left[\arraycolsep=\colsep\def\arraystretch{\stret}\begin{array}{cccccccc}
0&0&0&0&0&0&0&0\\0&0&1&1&1&1&1&1\\0&1&0&0&0&0&0&1\\0&1&0&0&1&1&0&1\\0&1&0&1&0&0&0&0\\0&1&0&1&0&0&1&1\\0&1&0&0&0&1&0&0\\0&1&1&1&0&1&0&0
\end{array}\right]
\left[\arraycolsep=\colsep\def\arraystretch{\stret}\begin{array}{cccccccc}
0&0&0&0&0&0&1&0\\0&0&0&0&0&0&1&1\\0&0&0&0&1&1&0&1\\0&0&0&0&1&1&0&0\\0&0&1&1&0&0&0&1\\0&0&1&1&0&0&0&0\\1&1&0&0&0&0&0&1\\0&1&1&0&1&0&1&0
\end{array}\right]
\left[\arraycolsep=\colsep\def\arraystretch{\stret}\begin{array}{cccccccc}
0&0&0&0&0&0&1&0\\0&0&0&0&0&1&0&1\\0&0&0&0&1&0&0&0\\0&0&0&0&1&1&1&1\\0&0&1&1&0&1&0&0\\0&1&0&1&1&0&1&0\\1&0&0&1&0&1&0&0\\0&1&0&1&0&0&0&0
\end{array}\right]
\left[\arraycolsep=\colsep\def\arraystretch{\stret}\begin{array}{cccccccc}
0&0&0&0&0&0&1&0\\0&0&0&0&0&1&1&1\\0&0&0&0&1&0&0&1\\0&0&0&0&1&1&0&0\\0&0&1&1&0&1&0&1\\0&1&0&1&1&0&0&1\\1&1&0&0&0&0&0&1\\0&1&1&0&1&1&1&0
\end{array}\right]
\vspace{3mm}
\\
\small 
\left[\arraycolsep=\colsep\def\arraystretch{\stret}\begin{array}{cccccccc}
0&0&0&0&0&0&1&0\\0&0&0&1&1&0&1&0\\0&0&0&1&0&1&0&1\\0&1&1&0&1&0&1&1\\0&1&0&1&0&1&1&0\\0&0&1&0&1&0&0&0\\1&1&0&1&1&0&0&0\\0&0&1&1&0&0&0&0
\end{array}\right]
\left[\arraycolsep=\colsep\def\arraystretch{\stret}\begin{array}{cccccccc}
0&0&0&0&0&0&1&0\\0&0&0&1&1&1&1&0\\0&0&0&1&0&1&0&1\\0&1&1&0&1&1&1&1\\0&1&0&1&0&0&1&0\\0&1&1&1&0&0&1&0\\1&1&0&1&1&1&0&0\\0&0&1&1&0&0&0&0
\end{array}\right]
\left[\arraycolsep=\colsep\def\arraystretch{\stret}\begin{array}{cccccccc}
0&0&0&0&0&0&1&0\\0&0&1&0&0&0&1&1\\0&1&0&1&0&1&1&1\\0&0&1&0&1&1&0&0\\0&0&0&1&0&0&0&1\\0&0&1&1&0&0&0&0\\1&1&1&0&0&0&0&1\\0&1&1&0&1&0&1&0
\end{array}\right]
\left[\arraycolsep=\colsep\def\arraystretch{\stret}\begin{array}{cccccccc}
0&0&0&0&0&0&1&0\\0&0&1&0&0&1&0&1\\0&1&0&1&0&0&1&0\\0&0&1&0&1&1&1&1\\0&0&0&1&0&1&0&0\\0&1&0&1&1&0&1&0\\1&0&1&1&0&1&0&0\\0&1&0&1&0&0&0&0
\end{array}\right]
\left[\arraycolsep=\colsep\def\arraystretch{\stret}\begin{array}{cccccccc}
0&0&0&0&0&0&1&0\\0&0&1&0&0&1&1&0\\0&1&0&1&0&0&0&0\\0&0&1&0&1&1&1&0\\0&0&0&1&0&1&1&1\\0&1&0&1&1&0&1&0\\1&1&0&1&1&1&0&0\\0&0&0&0&1&0&0&0
\end{array}\right]
\vspace{3mm}
\\
\small 
\left[\arraycolsep=\colsep\def\arraystretch{\stret}\begin{array}{cccccccc}
0&0&0&0&0&0&1&0\\0&0&1&0&0&1&1&1\\0&1&0&1&0&0&1&1\\0&0&1&0&1&1&0&0\\0&0&0&1&0&1&0&1\\0&1&0&1&1&0&0&1\\1&1&1&0&0&0&0&1\\0&1&1&0&1&1&1&0
\end{array}\right]
\left[\arraycolsep=\colsep\def\arraystretch{\stret}\begin{array}{cccccccc}
0&0&0&0&0&0&1&0\\0&0&1&1&1&0&0&0\\0&1&0&0&0&0&1&1\\0&1&0&0&1&0&1&0\\0&1&0&1&0&1&0&1\\0&0&0&0&1&0&0&1\\1&0&1&1&0&0&0&1\\0&0&1&0&1&1&1&0
\end{array}\right]
\left[\arraycolsep=\colsep\def\arraystretch{\stret}\begin{array}{cccccccc}
0&0&0&0&0&0&1&0\\0&0&1&1&1&0&1&0\\0&1&0&0&0&0&0&0\\0&1&0&0&1&0&1&1\\0&1&0&1&0&1&1&0\\0&0&0&0&1&0&0&0\\1&1&0&1&1&0&0&0\\0&0&0&1&0&0&0&0
\end{array}\right]
\left[\arraycolsep=\colsep\def\arraystretch{\stret}\begin{array}{cccccccc}
0&0&0&0&0&0&1&0\\0&0&1&1&1&1&0&0\\0&1&0&0&0&0&1&1\\0&1&0&0&1&1&1&0\\0&1&0&1&0&0&0&1\\0&1&0&1&0&0&1&1\\1&0&1&1&0&1&0&1\\0&0&1&0&1&1&1&0
\end{array}\right]
\left[\arraycolsep=\colsep\def\arraystretch{\stret}\begin{array}{cccccccc}
0&0&0&0&0&0&1&0\\0&0&1&1&1&1&1&0\\0&1&0&0&0&0&0&0\\0&1&0&0&1&1&1&1\\0&1&0&1&0&0&1&0\\0&1&0&1&0&0&1&0\\1&1&0&1&1&1&0&0\\0&0&0&1&0&0&0&0
\end{array}\right]
\vspace{3mm}
\\
\small 
\left[\arraycolsep=\colsep\def\arraystretch{\stret}\begin{array}{cccccccc}
0&0&0&0&0&0&1&0\\0&0&1&1&1&1&1&0\\0&1&0&0&0&0&1&1\\0&1&0&0&1&1&0&0\\0&1&0&1&0&0&0&1\\0&1&0&1&0&0&0&1\\1&1&1&0&0&0&0&0\\0&0&1&0&1&1&0&0
\end{array}\right]
\left[\arraycolsep=\colsep\def\arraystretch{\stret}\begin{array}{cccccccc}
0&0&0&0&0&0&1&0\\0&0&1&1&1&1&1&1\\0&1&0&0&0&0&1&1\\0&1&0&0&1&1&0&1\\0&1&0&1&0&0&0&0\\0&1&0&1&0&0&0&1\\1&1&1&0&0&0&0&1\\0&1&1&1&0&1&1&0
\end{array}\right]
\left[\arraycolsep=\colsep\def\arraystretch{\stret}\begin{array}{cccccccc}
0&0&0&0&0&0&1&1\\0&0&0&1&1&0&0&0\\0&0&0&1&0&1&1&1\\0&1&1&0&1&0&1&0\\0&1&0&1&0&1&0&0\\0&0&1&0&1&0&0&1\\1&0&1&1&0&0&0&0\\1&0&1&0&0&1&0&0
\end{array}\right]
\left[\arraycolsep=\colsep\def\arraystretch{\stret}\begin{array}{cccccccc}
0&0&0&0&0&0&1&1\\0&0&0&1&1&0&1&0\\0&0&0&1&0&1&0&1\\0&1&1&0&1&0&1&0\\0&1&0&1&0&1&1&0\\0&0&1&0&1&0&0&1\\1&1&0&1&1&0&0&1\\1&0&1&0&0&1&1&0
\end{array}\right]
\left[\arraycolsep=\colsep\def\arraystretch{\stret}\begin{array}{cccccccc}
0&0&0&0&0&0&1&1\\0&0&0&1&1&1&0&0\\0&0&0&1&0&1&1&1\\0&1&1&0&1&1&1&0\\0&1&0&1&0&0&0&0\\0&1&1&1&0&0&1&1\\1&0&1&1&0&1&0&0\\1&0&1&0&0&1&0&0
\end{array}\right]
\vspace{3mm}
\\
\small 
\left[\arraycolsep=\colsep\def\arraystretch{\stret}\begin{array}{cccccccc}
0&0&0&0&0&0&1&1\\0&0&1&0&0&1&0&1\\0&1&0&1&0&0&1&0\\0&0&1&0&1&1&1&0\\0&0&0&1&0&1&0&0\\0&1&0&1&1&0&1&1\\1&0&1&1&0&1&0&1\\1&1&0&0&0&1&1&0
\end{array}\right]
\left[\arraycolsep=\colsep\def\arraystretch{\stret}\begin{array}{cccccccc}
0&0&0&0&0&0&1&1\\0&0&1&0&0&1&1&1\\0&1&0&1&0&0&1&0\\0&0&1&0&1&1&0&0\\0&0&0&1&0&1&0&0\\0&1&0&1&1&0&0&1\\1&1&1&0&0&0&0&0\\1&1&0&0&0&1&0&0
\end{array}\right]
\left[\arraycolsep=\colsep\def\arraystretch{\stret}\begin{array}{cccccccc}
0&0&0&0&0&0&1&1\\0&0&1&1&1&0&0&0\\0&1&0&0&0&0&1&0\\0&1&0&0&1&0&1&0\\0&1&0&1&0&1&0&0\\0&0&0&0&1&0&0&1\\1&0&1&1&0&0&0&0\\1&0&0&0&0&1&0&0
\end{array}\right]
\left[\arraycolsep=\colsep\def\arraystretch{\stret}\begin{array}{cccccccc}
0&0&0&0&0&0&1&1\\0&0&1&1&1&1&0&0\\0&1&0&0&0&0&1&0\\0&1&0&0&1&1&1&0\\0&1&0&1&0&0&0&0\\0&1&0&1&0&0&1&1\\1&0&1&1&0&1&0&0\\1&0&0&0&0&1&0&0
\end{array}\right]
\left[\arraycolsep=\colsep\def\arraystretch{\stret}\begin{array}{cccccccc}
0&0&0&0&1&0&0&1\\0&0&0&1&0&1&0&0\\0&0&0&1&0&0&0&0\\0&1&1&0&1&0&1&1\\1&0&0&1&0&0&1&0\\0&1&0&0&0&0&1&1\\0&0&0&1&1&1&0&1\\1&0&0&1&0&1&1&0
\end{array}\right]
\vspace{3mm}
\\
\small 
\left[\arraycolsep=\colsep\def\arraystretch{\stret}\begin{array}{cccccccc}
0&0&0&0&1&0&0&1\\0&0&0&1&0&1&1&0\\0&0&0&1&0&0&0&0\\0&1&1&0&1&0&0&1\\1&0&0&1&0&0&1&0\\0&1&0&0&0&0&0&1\\0&1&0&0&1&0&0&0\\1&0&0&1&0&1&0&0
\end{array}\right]
\left[\arraycolsep=\colsep\def\arraystretch{\stret}\begin{array}{cccccccc}
0&0&0&0&1&0&0&1\\0&0&0&1&1&0&0&0\\0&0&0&1&1&1&0&0\\0&1&1&0&1&0&1&1\\1&1&1&1&0&0&0&0\\0&0&1&0&0&0&1&0\\0&0&0&1&0&1&0&1\\1&0&0&1&0&0&1&0
\end{array}\right]
\left[\arraycolsep=\colsep\def\arraystretch{\stret}\begin{array}{cccccccc}
0&0&0&0&1&0&0&1\\0&0&0&1&1&0&0&1\\0&0&0&1&1&1&0&1\\0&1&1&0&1&0&1&1\\1&1&1&1&0&0&0&0\\0&0&1&0&0&0&1&1\\0&0&0&1&0&1&0&0\\1&1&1&1&0&1&0&0
\end{array}\right]
\left[\arraycolsep=\colsep\def\arraystretch{\stret}\begin{array}{cccccccc}
0&0&0&0&1&0&1&1\\0&0&1&0&0&1&1&1\\0&1&0&1&1&0&1&0\\0&0&1&0&1&1&0&0\\1&0&1&1&0&0&0&1\\0&1&0&1&0&0&0&1\\1&1&1&0&0&0&0&0\\1&1&0&0&1&1&0&0
\end{array}\right]
\left[\arraycolsep=\colsep\def\arraystretch{\stret}\begin{array}{cccccccc}
0&0&0&1&0&1&1&0\\0&0&1&0&1&0&0&1\\0&1&0&0&1&0&0&1\\1&0&0&0&0&1&1&0\\0&1&1&0&0&0&0&1\\1&0&0&1&0&0&1&0\\1&0&0&1&0&1&0&0\\0&1&1&0&1&0&0&0
\end{array}\right]
\\
\caption{Adjacency matrices \(B_6,\dots,B_{40}\in\mathcal{B}_{R_\text{cube}}\).}
    \label{tab:cubes}
\end{longtable}

Together with Lemma~\ref{lem:cube}, this can be summarized as follows.

\begin{theorem}[Cube switching]\label{thm:cube}
    Let \(\Gamma\) be a graph and \(C=\{v_1,\dots,v_8\}\) a subset of its vertices, identified with the eight points of the affine space \(\mathrm{AG}(3,2)\), so the vertices of a cube. Suppose that: 
    \begin{enumerate}[(i)]
        \item The induced subgraph on \(C\) is empty or one of the graphs below:\\ \newcommand{\radius}{.8}
\newcommand{\shift}{.4}
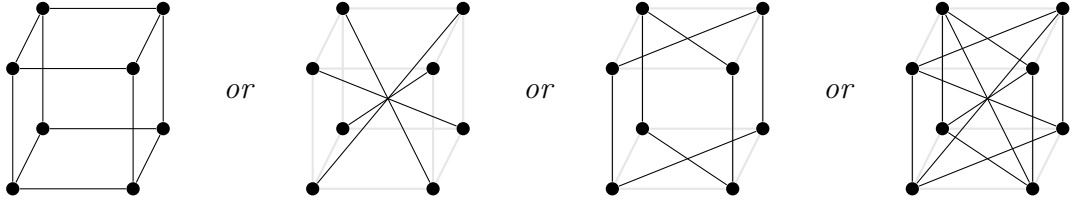
\begin{tikzpicture}[baseline=(O.base)]
\path (1.5*\radius-.5*\shift,1.5*\radius) node (O) {};
\path (-.5*\radius-.5*\shift,-.5*\radius) rectangle (3.5*\radius-.5*\shift,3.5*\radius);
\path[every node/.append style={circle, fill=black, minimum size=5pt, label distance=0pt, inner sep=0pt}]
(2*\radius,0) node (3) {}
(2*\radius,2*\radius) node (6) {}
(3*\radius-\shift,\radius) node (2) {}
(3*\radius-\shift,3*\radius) node (7) {}
(0,0) node (8) {}
(0,2*\radius) node (1) {}
(\radius-\shift,\radius) node (5) {}
(\radius-\shift,3*\radius) node (4) {};
    \draw (1) edge (4) edge (6) edge (8)
    (3) edge (2) edge (6) edge (8)
    (5) edge (2) edge (4) edge (8)
    (7) edge (2) edge (4) edge (6);
\end{tikzpicture} or 
\begin{tikzpicture}[baseline=(O.base)]
    \path (1.5*\radius-.5*\shift,1.5*\radius) node (O) {};
    \path (-.5*\radius-.5*\shift,-.5*\radius) rectangle (3.5*\radius-.5*\shift,3.5*\radius);
    \path[every node/.append style={circle, fill=black, minimum size=5pt, label distance=0pt, inner sep=0pt}]
    (2*\radius,0) node (101) {}
    (2*\radius,2*\radius) node (001) {}
    (3*\radius-\shift,\radius) node (111) {}
    (3*\radius-\shift,3*\radius) node (011) {}
    (0,0) node (100) {}
    (0,2*\radius) node (000) {}
    (\radius-\shift,\radius) node (110) {}
    (\radius-\shift,3*\radius) node (010) {};
    \draw[thick,backline] (000) edge (100) edge (010) edge (001)
    (011) edge (111) edge (001) edge (010)
    (101) edge (001) edge (111) edge (100)
    (110) edge (010) edge (100) edge (111);
    \draw (000) edge (111)
    (100) edge (011)
    (101) edge (010)
    (110) edge (001);
\end{tikzpicture} or 
\begin{tikzpicture}[baseline=(O.base)]
    \path (1.5*\radius-.5*\shift,1.5*\radius) node (O) {};
    \path (-.5*\radius-.5*\shift,-.5*\radius) rectangle (3.5*\radius-.5*\shift,3.5*\radius);
    \path[every node/.append style={circle, fill=black, minimum size=5pt, label distance=0pt, inner sep=0pt}]
    (2*\radius,0) node (101) {}
    (2*\radius,2*\radius) node (001) {}
    (3*\radius-\shift,\radius) node (111) {}
    (3*\radius-\shift,3*\radius) node (011) {}
    (0,0) node (100) {}
    (0,2*\radius) node (000) {}
    (\radius-\shift,\radius) node (110) {}
    (\radius-\shift,3*\radius) node (010) {};
    \draw[thick,backline] (000) edge (100) edge (010) edge (001)
    (011) edge (111) edge (001) edge (010)
    (101) edge (001) edge (111) edge (100)
    (110) edge (010) edge (100) edge (111);
    \draw (000) edge (100) edge (011)
    (111) edge (100) edge (011)
    (101) edge (001) edge (110)
    (010) edge (001) edge (110);
\end{tikzpicture} or \begin{tikzpicture}[baseline=(O.base)]
    \path (1.5*\radius-.5*\shift,1.5*\radius) node (O) {};
    \path (-.5*\radius-.5*\shift,-.5*\radius) rectangle (3.5*\radius-.5*\shift,3.5*\radius);
    \path[every node/.append style={circle, fill=black, minimum size=5pt, label distance=0pt, inner sep=0pt}]
    (2*\radius,0) node (101) {}
    (2*\radius,2*\radius) node (001) {}
    (3*\radius-\shift,\radius) node (111) {}
    (3*\radius-\shift,3*\radius) node (011) {}
    (0,0) node (100) {}
    (0,2*\radius) node (000) {}
    (\radius-\shift,\radius) node (110) {}
    (\radius-\shift,3*\radius) node (010) {};
    \draw[thick,backline] (000) edge (100) edge (010) edge (001)
    (011) edge (111) edge (001) edge (010)
    (101) edge (001) edge (111) edge (100)
    (110) edge (010) edge (100) edge (111);
    \draw (000) edge (100) edge (011)
    (111) edge (100) edge (011)
    (101) edge (001) edge (110)
    (010) edge (001) edge (110);
    \draw (000) edge (111)
    (100) edge (011)
    (101) edge (010)
    (110) edge (001);
\end{tikzpicture}
        or a graph whose adjacency matrix is one of the matrices \(B_6,\dots,B_{40}\) of Table~\ref{tab:cubes}, embedded on a cube with the same labelling as in Lemma~\ref{lem:cube}.
        \item Every vertex outside \(C\) is adjacent to either:
        \begin{enumerate}[1.]
            \item All vertices of \(C\).
            \item No vertices of \(C\).
            \item Four vertices of \(C\) contained in an affine plane.
        \end{enumerate}
    \end{enumerate}
    \renewcommand*{\thefootnote}{*}
    Let \(\pi\) be an automorphism of the cube of order six\footnote{In other words, a ``rotation + complementation''. Up to an automorphism of the switching set, the choice of \(\pi\) does not matter.}
    . For every \(v\) outside \(C\) that is adjacent to the four vertices of a plane
    , make it adjacent to the four vertices of the plane obtained by the following rules:
    \begin{itemize}
    \item A face of the cube is mapped to its image under \(\pi\).
    \item A plane connecting two opposite edges, is taken to its complement (as a set of vertices).
    \item A set of four non-adjacent vertices on the cube is fixed.
    \end{itemize}
    If the induced subgraph on \(C\) has an adjacency matrix \(B\) from Table~\ref{tab:cubes}, replace it by the graph with adjacency matrix \(R_\text{cube}^TBR_\text{cube}\). The resulting graph is \(\mathbb{R}\)-cospectral with \(\Gamma\).
\end{theorem}

\begin{proof}
     Using the same labelling as in Lemma~\ref{lem:cube}, the adjacency matrices are given by \(B_1=O\),
    \small 
    \[\renewcommand{\colsep}{1.4pt}
    \renewcommand{\stret}{.6}
    B_2\!=\!\left[\arraycolsep=\colsep\def\arraystretch{\stret}\begin{array}{cccccccc}
    0&0&0&1&0&1&0&1\\0&0&1&0&1&0&1&0\\0&1&0&0&0&1&0&1\\1&0&0&0&1&0&1&0\\0&1&0&1&0&0&0&1\\1&0&1&0&0&0&1&0\\0&1&0&1&0&1&0&0\\1&0&1&0&1&0&0&0
    
    \end{array}\right]\!,\;
    B_3\!=\!\left[\arraycolsep=\colsep\def\arraystretch{\stret}\begin{array}{cccccccc}
    0&1&0&0&0&0&0&0\\1&0&0&0&0&0&0&0\\0&0&0&1&0&0&0&0\\0&0&1&0&0&0&0&0\\0&0&0&0&0&1&0&0\\0&0&0&0&1&0&0&0\\0&0&0&0&0&0&0&1\\0&0&0&0&0&0&1&0
    \end{array}\right]\!,\;
    B_4\!=\!\left[\arraycolsep=\colsep\def\arraystretch{\stret}\begin{array}{cccccccc}
    0&0&0&0&0&0&1&1\\0&0&0&0&0&0&1&1\\0&0&0&0&1&1&0&0\\0&0&0&0&1&1&0&0\\0&0&1&1&0&0&0&0\\0&0&1&1&0&0&0&0\\1&1&0&0&0&0&0&0\\1&1&0&0&0&0&0&0
    \end{array}\right]\!,\;
    B_5\!=\!\left[\arraycolsep=\colsep\def\arraystretch{\stret}\begin{array}{cccccccc}
    0&1&0&0&0&0&1&1\\1&0&0&0&0&0&1&1\\0&0&0&1&1&1&0&0\\0&0&1&0&1&1&0&0\\0&0&1&1&0&1&0&0\\0&0&1&1&1&0&0&0\\1&1&0&0&0&1&0&1\\1&1&0&0&0&0&1&0
    \end{array}\right]\]
    and the matrices $B_6,\dots,B_{40}$ in Table~\ref{tab:cubes}.
    One checks that \(R_\text{cube}^TB_iR_\text{cube}=B_i\) if \(i\in\{1,2,3,4,5\}\). The conditions on the vertices outside \(C\) are given by Lemma~\ref{lem:cube}.
\end{proof}


The matrices \(B_1,\dots,B_{40}\) give a full (algebraic) description of Cube switching, completing the work started in \cite[Section~6]{AHswitching}.

\begin{figure}[H]
    \centering
\begin{subfigure}{.2\textwidth}
\centering
\newcommand{\radius}{.8}
\newcommand{\shift}{.3}
\begin{tikzpicture}[baseline=(O.base), scale=.8]
    \path (1.5*\radius-.5*\shift,1.5*\radius) node (O) {};
    \path (-.5*\radius,-.5*\radius) rectangle (3.5*\radius,3.5*\radius);
    \path[every node/.append style={circle, fill=black, minimum size=5pt, label distance=0pt, inner sep=0pt}]
    (2*\radius,0) node (101) {}
    (2*\radius,2*\radius) node (001) {}
    (3*\radius-\shift,\radius) node (111) {}
    (3*\radius-\shift,3*\radius) node (011) {}
    (0,0) node (100) {}
    (0,2*\radius) node (000) {}
    (\radius-\shift,\radius) node (110) {}
    (\radius-\shift,3*\radius) node (010) {}
    (-\radius-.5*\shift,1.6*\radius) node (L) {}
    (4*\radius-.5*\shift,1.3*\radius) node (R) {};
    \draw (000) edge (100) edge (010) edge (001)
    (011) edge (111) edge (001) edge (010)
    (101) edge (001) edge (111) edge (100)
    (110) edge (010) edge (100) edge (111)
    (L) edge (000) edge (010) edge (101) edge (111)
    (R) edge (001) edge (011) edge (101) edge (111);
\end{tikzpicture}
\end{subfigure}
\qquad
\begin{subfigure}{.2\textwidth}
\centering
\newcommand{\radius}{.8}
\newcommand{\shift}{.3}
\begin{tikzpicture}[baseline=(O.base), scale=.8]
    \path (1.5*\radius-.5*\shift,1.5*\radius) node (O) {};
    \path (-.5*\radius,-.5*\radius) rectangle (3.5*\radius,3.5*\radius);
    \path[every node/.append style={circle, fill=black, minimum size=5pt, label distance=0pt, inner sep=0pt}]
    (2*\radius,0) node (101) {}
    (2*\radius,2*\radius) node (001) {}
    (3*\radius-\shift,\radius) node (111) {}
    (3*\radius-\shift,3*\radius) node (011) {}
    (0,0) node (100) {}
    (0,2*\radius) node (000) {}
    (\radius-\shift,\radius) node (110) {}
    (\radius-\shift,3*\radius) node (010) {}
    (-\radius-.5*\shift,1*\radius) node (L) {}
    (4*\radius-.5*\shift,2*\radius) node (R) {};
    \draw (000) edge (100) edge (010) edge (001)
    (011) edge (111) edge (001) edge (010)
    (101) edge (001) edge (111) edge (100)
    (110) edge (010) edge (100) edge (111)
    (L) edge (001) edge (011) edge (100) edge (110)
    (R) edge (010) edge (011) edge (110) edge (111);
\end{tikzpicture}
\end{subfigure}
\caption{Cospectral mates, obtained by Cube switching.}
\end{figure}
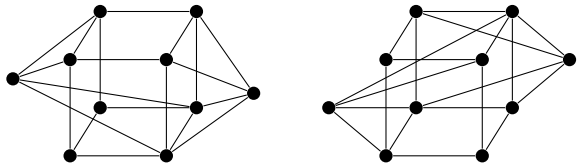

\section{Concluding remarks}\label{sec:concludingremarks}

We proposed two new switching methods of level 2 to construct \(\mathbb{R}\)-cospectral graphs (see Theorem~\ref{thm:infinite1} and Theorem~\ref{thm:infinite2}). We also classified all irreducible switching methods of level 2 with one non-trivial indecomposable block and switching set size up to size 12, extending the results from \cite{AHswitching,MAO2023}. They are given by:
\begin{itemize}
    \item GM-switching (Theorem~\ref{thm:GM}) on 4 vertices,
    \item Six vertex AH-switching (Theorem~\ref{thm:AH6}) on 6 vertices,
    \item Fano switching (Theorem~\ref{thm:fano}) on 7 vertices,
    \item Ten vertex switching (Corollary~\ref{cor:10vertex}) on 10 vertices,
    \item Twelve vertex switching (Theorem~\ref{thm:12vertex}) on 12 vertices.
\end{itemize}

\noindent We also provided alternative and more intuitive combinatorial interpretations of Sun graph switching or MWLQ-switching (Theorem~\ref{thm:sun}), Six vertex AH-switching (Theorem~\ref{thm:AH6}), as well as geometrical descriptions of Seven vertex AH-switching (Fano switching, see Theorem~\ref{thm:fano}) and of the sporadic Eight vertex AH-switching (Cube switching, see Theorem~\ref{thm:cube}). Thanks to the geometrical interpretation of Fano switching, we derived an alternative and shorter proof of the fact that 2-Kneser graphs are not determined by their spectrum (Theorem~\ref{thm:2kneser}), which was originally shown in \cite{johnson} using GM-switching.

As future work, it would be interesting to see more applications of the newly introduced methods of level 2, for instance in the construction of new strongly regular graphs, similarly as is done in \cite{abiad2016switched,twisted,barwick}, among others.
Regarding Theorem~\ref{thm:sunirred}, it remains open whether the new infinite families from Theorems~\ref{thm:infinite1} and \ref{thm:infinite2} are irreducible, just like we proved for Sun graph switching. Furthermore, it remains an open problem to investigate switching methods of level 2 from matrices with multiple non-trivial indecomposable blocks. If \(t\geq2\), then condition (i) in the statement of GM-switching (Theorem~\ref{thm:GM}) is stricter than necessary. Thus, one could try to determine a necessary condition that is also sufficient. 
More generally, it would be desirable to extend the classification from Section \ref{sec:classification} to methods coming from matrices with \emph{multiple} non-trivial indecomposable blocks. Such an extension should also include a generalization of the definition of reducibility to these methods. 

\subsection*{Acknowledgements}

Aida Abiad and Nils van de Berg are supported by the Dutch Research Council (NWO) through the grant VI.Vidi.213.085. Robin Simoens is supported by the Research Foundation Flanders (FWO) through the grant 11PG724N.

\printbibliography[heading=bibintoc]

\newpage
\appendix

\section*{Appendix}\label{app:code}


Here we describe how to find the irreducible matrices for a given switching method of level 2. The code for this can be found on \url{https://github.com/robinsimoens/level-2-switching}. We use
the mathematics software system SageMath \cite{sagemath}.

Suppose that we want to find the irreducible matrices for a given switching method, corresponding to a regular orthogonal matrix \(R\) of level 2. We apply the following algorithm.
\begin{enumerate}[leftmargin=*]
    \item First, we determine all elements of \(\mathcal{B}_R\). For Fano switching and Cube switching, this is feasible in a straightforward way, by looping through all symmetric 01-matrices \(B\) and checking whether \(R^TBR\) is again a 01-matrix. For the switching methods of the infinite family from Theorem~\ref{thm:level2mats}(ii), we use equation~\eqref{eq:1} to first determine the possible \(4\times4\) submatrices consisting of four \(2\times2\) blocks, and then ``patching'' them together to obtain all possible matrices \(B\in\mathcal{B}_R\). This strategy is, simultaneously and independently, also used by Mao and Yan \cite{MaoSimilar}. In this step, we already use Lemma~\ref{lem:even} and Lemma~\ref{lem:complement} to restrict to those blocks with an even number of ones, up to complementation.
    \item Since \(\mathcal{B}_R\) is invariant under complementation and under certain conjugations (see Lemma \ref{lem:symmetry}), we calculate the ``stabiliser'' of the matrix \(R\) and use it to choose one representative for each conjugacy class of matrices in \(\mathcal{B}_R\).
    \item Finally, we check which of the remaining matrices are reducible. We first create a list of all orthogonal matrices whose largest indecomposable block is smaller in size. For Twelve vertex switching, we have to be extra careful that there are two types of matrices with two indecomposable blocks of size 6, since Six vertex AH-switching is ``directed''. Then, for each matrix \(B\), we check if there is a sequence of such smaller matrices that maps \(B\) to other adjacency matrices under conjugation in every step, and whose product is equal to \(R\), up to a permutation of the columns (since the graphs are then isomorphic).

    The smaller switching sets can be:
    \begin{itemize}
        \item For the infinite family, only matrices of the infinite family again, see Lemma~\ref{lem:infreducible}.
        \item Fano switching may be reducible to GM-switching or Six vertex AH-switching. GM-switching sets correspond to four points of the Fano plane such that every line intersects it in an even number of points. In other words, the complement of a line. Six vertex AH-switching sets correspond to three pairs of points such that every line intersects them in the same number of points modulo 2. In other words, a set of three lines through a fixed point, without that point.
        \item Cube switching may be reducible to GM-switching, Six vertex AH-switching or Fano switching. 
        In contrast to checking the reducibility of the infinite family and Fano switching, one needs to be extra careful when checking for smaller switching sets, since Fano switching ``messes up'' the adjacencies of the vertices outside the switching set (since three points on a line go to three points not on a line). Therefore, in every step, we run through all possible subsets, and permutations of them, to check if switching is possible. In particular, we check if it agrees with the adjacencies of the vertices outside the Cube switching set.
    \end{itemize}
\end{enumerate}
We used this algorithm to prove Theorem~\ref{thm:AH8infreducible}, Theorem~\ref{thm:AH10}, Theorem~\ref{thm:12vertex}, Theorem~\ref{thm:fano} and Theorem~\ref{thm:AH8sporreducible} and to obtain the matrices in Table~\ref{tab:AH12} and Table~\ref{tab:cubes}.

\end{document}